\documentclass{amsart}
\usepackage{tikz}
\usepackage{hyperref}   

\usepackage{amssymb,xcolor}
\usepackage{stmaryrd} 

\usepackage{mathtools}  
\mathtoolsset{showonlyrefs}

\theoremstyle{plain}
\newtheorem{theorem}{Theorem}[section]
\newtheorem{corollary}[theorem]{Corollary}
\newtheorem{definition}[theorem]{Definition}

\theoremstyle{remark}
\newtheorem{remark}[theorem]{Remark}
\theoremstyle{plain}
\newtheorem{proposition}[theorem]{Proposition}
\numberwithin{equation}{section}
\theoremstyle{plain}
\newtheorem{lemma}[theorem]{Lemma}
\theoremstyle{plain}

\theoremstyle{plain}
\newtheorem{prop}[theorem]{Proposition}

\theoremstyle{definition}

\newtheorem{defn}[theorem]{Definition}

\def\p{\partial}

\def\b{\bar}
\def\mb{\mathbb}
\def\mc{\mathcal}

\def\n{\nabla}

\begin{document}

\title{Scalar curvature rigidity of domains in a warped product}

\author{Xiaoxiang Chai}
\address{Department of Mathematics, POSTECH, 77 Cheongam-Ro, Nam-Gu, Pohang, Gyeongbuk, Korea 37673}
\email{xxchai@kias.re.kr, xxchai@postech.ac.kr}

\author{Xueyuan Wan}
\address{Mathematical Science Research Center, Chongqing University of Technology,\newline Chongqing 400054, China}
\email{xwan@cqut.edu.cn}

\keywords{Scalar curvature, mean curvature, dihedral rigidity, spinor, Dirac operator.}
\subjclass[2020]{53C24, 53C27, 53C21}

\begin{abstract}
  By exploiting the conformality of a warped product metric with a direct
product metric, we develop a new connection on a twisted spinor bundle
and its associated Dirac operator. We obtain a Llarull type scalar curvature
rigidity for a general class of domains in a warped product.
Also, we are able to address Gromov dihedral rigidity in hyperbolic space assuming matching angles.
\end{abstract}

\maketitle

\section{Introduction}

Given a manifold $(M, g)$ with boundary, we denote $\nu := \nu_g$ the inward
pointing unit normal vector with respect to $g$. We denote by $\nabla^g$ by the
Levi-Civita connection of $g$ and $R_g$ the scalar curvature of $g$, $A_g = -
\nabla^g \nu$ the second fundamental form of $\partial M$ with respect to $g$,
$H_g$ the mean curvature of $\partial M$ in $M$ with respect to $g$ which is
the trace of the second fundamental form.

Let $n > 2$, $r_- < r_+$
and $\varphi$ be a positive smooth function on $[r_-, r_+]$. A
\text{{\itshape{warped product metric}}} is a metric of the form
\begin{equation}
  g_{0} = \mathrm{d} r^2 + \varphi (r)^2 h, \label{warped product of X}
\end{equation}
on some manifold $X$ of dimension $n - 1$ and $h$ is a Riemannian metric on
$X$. For simplicity, we also require that $\varphi (r_{\pm}) \neq 0$. By a direct
calculation, the scalar curvature of $g_{0}$ on $M := [r_-, r_+] \times X$ is
given by
\[ R_{g_{0}} = \frac{1}{\varphi^2} R_{h} -n(n-1)(\frac{\varphi'}{\varphi})^2-2(n-1)(\frac{\varphi'}{\varphi})',\]
with the boundary mean curvature of $\partial_{\pm} M := \{r_{\pm}\} \times X$ is given by
\[ H_{g_{0}} = \pm (n - 1) \varphi (r_{\pm})^{- 1} \varphi' (r_{\pm}) \text{ along } \partial_{\pm}M. \]
Concerning scalar curvature rigidity of $M$, there are various works, and we refer the readers to the 
\text{{\itshape{Four lectures on scalar curvature}}} of Gromov \cite{gromov-four-2021}
for a wealth of such results. In this paper, we are interested in the scalar curvature of rigidity for domains in warped products. In particular, we study Llarull theorem
\cite{llarull-sharp-1998} for domains in warped products and the
Gromov dihedral rigidity conjecture in hyperbolic space.

The Llarull theorem \cite{llarull-sharp-1998} (see also \cite{goette-scalar-2002}, \cite{listing-scalar-arxiv-2010}) asserts that a
metric on the sphere and the scalar curvature cannot be at the same time
bounded below by those of the standard sphere metric.
A distinct feature of Llarull theorem compared to the scalar curvature rigidity of torus (see \cite{schoen-existence-1979, gromov-positive-1983}) is the requirement of a metric comparison.
There are multiple proofs of this scalar curvature rigidity including \cite{llarull-sharp-1998}, \cite{hirsch-rigid-2022-arxiv}, \cite{hu-rigidity-2023}, \cite{bar-scalar-2024}, \cite{wang-scalar-arxiv-2023}, \cite{li-spectral-2024}. Some proofs of these papers apply to a warped product metric of a closed manifold of positive curvature and a closed interval. 


Lott \cite{lott-index-2021} generalized Llarull theorem with \textit{convex} boundary which includes the effect of the mean curvature with subsequent developments by \cite{wang-dihedral-2023-arxiv}. To state our first result which is a Llarull theorem for domains in a warped product, we observe that every warped product in \eqref{warped product of X} is conformal to the
direct product metric
\begin{equation}
  \bar{g} = \mathrm{d} s^2 + g_X \label{direct product of X}
\end{equation}
where $s$ is a function of $r \in [r_-, r_+]$ given by
\[ s (r) = \int_{r_-}^{r_{+}} \varphi (t)^{- 1} \mathrm{d} t . \]
Indeed, since $s (r)$ is monotone with respect to $r$, $r$ can also be seen implicitly
as a function of $s \in [0, \int_{r_-}^{r_+} \varphi (t)^{- 1} \mathrm{d}
t]$. We set $\psi (s) = \varphi (r(s))$, and it is easy to check that $g_{0} =
\psi (s)^2 \bar{g}$.

Now our version of Llarull theorem with boundary takes the following form.

\begin{theorem}
  \label{general Llarull}
  Let $(X \times I, g_0 = dr^2 + \varphi(r)^2 h)$ be an $n$-dimensional warped product manifold such that the curvature operator of $(X, h)$ is non-negative and $\varphi$ is log-concave, that is,
  \begin{equation}
    (\log \varphi)'' \leq 0. \label{log concavity}
    \end{equation}
	Let $M$ be an $n$-dimensional compact manifold in $(X \times I, g_{0} = dr^2 + \varphi(r)^2 h) = (X \times I_s, \psi(s)^2 \bar{g})$ such that the boundary $\partial M$ is convex with respect to $\bar{g}$ and the Euler characteristic
 of $M$ is non-zero. 
	If $(N, {g})$ is a compact Riemannian manifold and $f: N \to M$ is a spin map such that
	\begin{enumerate}
		\item[(i)] $R_{{g}} \geq f^{\ast} R_{g_{0}}$ in $N$,
		\item[(ii)] $H_{{g}} \geq f^{\ast} H_{g_{0}}$ on $\partial N$,
		\item[(iii)] $f: (N, {g}) \to (M, g_{0})$ is distance non-increasing,
		\item[(iv)] the degree of $f$ is non-zero,
	\end{enumerate}
then the equalities in (i) and (ii) hold. Moreover, $f: (N, {g}) \to (M, g_{0})$ is a local isometry if $\varphi$ is strictly log-concave and $(X, h)$ is Ricci positive, and $\partial f=f|_{\p N}: (\partial N, {g}) \to (\partial M, g_{0})$ is a local isometry if $\partial M$ is strictly convex with respect to $g_0$.
\end{theorem}

In principle, Theorem \ref{general Llarull} could be
generalized to manifolds with polyhedral boundary as well. These generalizations should be
straightforward using
the machinery of index theory for manifolds with corners
in \cite{wang-gromovs-2022-arxiv}, \cite{wang-dihedral-2023-arxiv}. Hence, we only consider smooth domains.
A related result that could be compared to ours was given by Wang-Xie \cite{wang-dihedral-2023-arxiv}.
They considered the
scalar curvature rigidity of
radially convex domains (see \cite[Definition 1.1]{wang-dihedral-2023-arxiv}) with polyhedral boundary in a warped product. 

A warped product of particular interest is given by
\[ \mathrm{d} r^2 + \varphi (r)^2 (\mathrm{d} t^2 + \phi (t)^2 g_Y) \]
on $[r_-, r_+] \times [0, t_0] \times Y$ where $X= [0,t_{0}]\times Y$ and its metric $h= \mathrm{d}t^2 + \phi(t)^2 g_{Y}$. We also require that the metric
$\mathrm{d} t^2 + \phi (t)^2 g_Y$ on $[0, t_0] \times Y$ is complete at $t
= 0$, the curvature operator of $\mathrm{d} t^2 + \phi (t)^2 g_Y$ is
positive and $\phi' / \phi > 0$ for all $t \in (0, t_0)$. We take a
domain $\Omega$ in $[r_-, r_+] \times [0, t_0] \times Y$ given by
\[ \Omega = \cup_{r \in [r_-, r_+]} \Sigma_r : = \cup_{r \in [r_-, r_+]} \{(r,
   t, y) \in [r_-, r_+] \times [0, t_0] \times Y : \text{ } t \leq \tau 
   (r)\} \]
for some positive function $\tau$ on $[r_-, r_+]$. This domain $\Omega$ is \textit{rotationally symmetric} with respect to the $r$ direction. Let
\[ \partial_s \Omega = \cup_{r \in [r_-, r_+]} \{(r, t, y) \in [r_-, r_+]
   \times [0, t_0] \times Y : \text{ } t = \tau (r)\} . \]
 It is easy to check that the inward unit normal of $\partial_s \Omega$ is
 \[\nu = (\rho' \varphi^2 \partial_r - \partial_t) \varphi^{-1} ((\rho'\varphi)^2+1)^{-1/2}.\]
 The dihedral angles $\gamma$ formed
by $\partial_s \Omega$ and $\Sigma_r$ is then
\[ \cos \gamma = - \langle \nu, \partial_r \rangle = - \frac{\rho'
   \varphi}{\sqrt{(\rho' \varphi)^2 + 1}} . \]
Note that the angle is conformally invariant and it is easy to check that the condition $\partial_s \Omega$
is convex with respect to $\bar{g} = \mathrm{d} r^2 + \mathrm{d} t^2 + \phi
(t)^2 g_Y$ is equivalent to that $\gamma$ (or $\rho' \varphi$) is decreasing with
respect to $r$ and the condition $\phi'/\phi>0$ which we have already assumed. In fact, the condition
\begin{equation}
  \gamma'(r)<0 \label{boundary analog}
  \end{equation}
is the \textit{boundary analog of logarithmic concavity \eqref{log concavity}}. The condition \eqref{boundary analog} answers
  a problem raised by Gromov in the settings of capillary surfaces (see \cite[Section 5.8.1]{gromov-four-2021}). \footnote{In a work in preparation, the first author with Gaoming Wang also found this condition in the settings of capillary surfaces.} See Figure \ref{rot picture} for an illustration with $\varphi(r)=r^2$, $\phi(t)=\sin t$ which is essentially the flat metric of $\mathbb{R}^3$ with a radial foliation.
\begin{figure}
\centering
\includegraphics{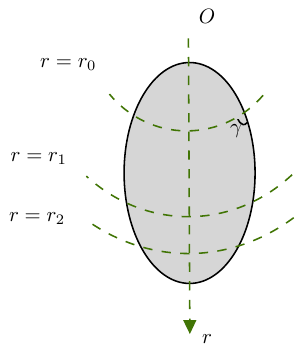}
\caption{A rotationally symmetric domain with a radial foliation in $\mathbb{R}^3$.}
\label{rot picture}
\end{figure}

As another special case, we have the following Llarull type rigidity in hyperbolic
case, that is, when $X =\mathbb{R}^{n - 1}$, $g_X$ is the standard flat
metric, $\psi (s) = s^{- 1}$ and the map $f$ is the
identity.
This is the Poincar{\'e} half-space model of the hyperbolic $n$-space, which we now recall.
The model is given by the
metric
\[ b = \tfrac{1}{(x^1)^2} ((\mathrm{d} x^1)^2 + \cdots + (\mathrm{d} x^n)^2)
\]
on $\mathbb{R}_+^n = \{(x^1, \ldots, x^n) : x^1 > 0\}$.
By
convention, we use $\delta$ to denote the flat metric on $\mathbb{R}^n_+$.

\begin{theorem}\label{hyperbolic Llarull}
Let $\Omega$ be a compact, strictly convex, smooth domain in $(\mb{R}^n_+,\delta)$. Let $b=\frac{1}{(x^1)^2}\delta$ be the hyperbolic metric on $\mb{R}^n_+$. Suppose $g$ is a Riemannian metric on $\Omega$ which satisfies
\begin{itemize}
\item[(1)] The scalar curvature $R_g \geq -n(n-1)$;
\item[(2)] The mean curvature on the boundary $H_g \geq H_b$;
\item[(3)] The induced metrics $\left.g\right|_{\partial \Omega} \geq \left.b\right|_{\partial \Omega}$.
\end{itemize}
Then $(\Omega, g)$ is hyperbolic and the induced metrics of $g$ and $b$ on $\partial \Omega$ agree.
\end{theorem}
This theorem answers a question of Gromov  (see \cite[Section 3.5]{gromov-four-2021}; On Non-spin Manifolds and on $\sigma<0$) concerning the scalar curvature rigidity of geodesic balls in hyperbolic space. 
The reason that we do not require in the hyperbolic case that $g
\geq b$ (equivalently, the identity map is distance non-increasing) is
that the hyperbolic space is conformal to the Euclidean space which is
flat.  We take the idea further by considering $\partial M$ being totally
geodesic with respect to the flat metric. We can then remove the third
condition on induced metrics on the boundary, however, we have to then consider a polytope. This is the
hyperbolic version of the Gromov dihedral rigidity conjecture.

We are interested in the scalar curvature rigidity of
polytopes enclosed by linear hyperplanes, which are umbilic in the Poincar\'{e} half-space model.
Let $\Omega$ be a compact, convex polytope in $(\mathbb{R}^n_+, \delta)$ with
non-empty interior, we may write $\Omega = \cap_{\ell\in\Lambda } \{u_\ell \leq
0\}$, where $u_\ell$, $\ell \in \Lambda$ is a finite collection of non-constant linear
functions defined on $\mathbb{R}^n_+$. For each $\ell \in \Lambda$, we denote by $N_\ell
\in \mathbb{S}^{n - 1}$ the outward-pointing unit normal vector to the half-space $\{u_\ell \leq 0\}$ with respect to the Euclidean metric.

Let $g$ be another Riemannian metric which is defined on an open set containing
$\Omega$. For each $\ell \in \Lambda$, we denote by $\nu_\ell$ the outward unit normal
vector to the half-space $\{u_\ell \leq 0\}$ with respect to the metric $g$.
For adjacent faces $F_{\ell_{i}}, F_{\ell_{j}} \subset \Omega$, we call the angle $\gamma_{i j}
\in (0, \pi)$ with $\cos \gamma_{i j} = - \cos \langle \nu_{\ell_{i}}, \nu_{\ell_{j}} \rangle$ the
dihedral angle. We add a bar to $\gamma$ to indicate that the angle is
computed with respect to the flat metric.

\begin{theorem}
  \label{intersect rigidity}Let $\Omega$ be
  a convex polytope in $(\mathbb{R}^n_+,\delta)$,  $g$ be another Riemannian metric on $\Omega$ with
  scalar curvature $R_g \geq - n (n - 1)$ and each face $F_{\ell}$ and $\ell
  \in \Lambda$ has mean curvature $H_{\ell} \geq (n - 1) \langle
  \tfrac{\partial}{\partial x^1}, N_{\ell} \rangle_{\delta}$. 
If $p$ is a point in $\partial \Omega$ and $\ell_1, \ell_2 \in \Lambda$
satisfying $u_{\ell_1} (p) = u_{\ell_2} (p) = 0$, then $g (\nu_{\ell_1},
\nu_{\ell_2}) = \langle N_{\ell_1}, N_{\ell_2} \rangle_{\delta}$ at the point
$p$ (matching angle).
  Then $(\Omega, g)$ is
  hyperbolic with umbilic faces $F_{\ell}$ whose mean curvatures are given by
  $(n - 1) \langle \tfrac{\partial}{\partial x^1}, N_{\ell} \rangle_{\delta}$. Moreover, $(\Omega,g)$ lies in a Poincar\'{e}
  half-space model.
\end{theorem}

This theorem is a special case of Gromov's dihedral rigidity
conjecture \cite{gromov-dirac-2014} in hyperbolic space. Gromov
\cite{gromov-dirac-2014} raised the dihedral rigidity conjecture
to characterize the manifold with weaker notions of scalar curvature
lower bound, and he proved the Euclidean version of the
conjecture for cubes, later Li \cite{li-polyhedron-2020} first
proved for some special polyhedra (prisms) other than cubes in
dimension 3. Li's method is by capillary minimal surface.
We have also seen
several developments using spinors \cite{wang-gromovs-2022-arxiv},
\cite{wang-dihedral-2023-arxiv}, \cite{brendle-scalar-2024},
and \cite{brendle-gromovs-arxiv-2023}.

In the hyperbolic case, the dihedral rigidity originally was only raised
for parabolic cubes in \cite{gromov-dirac-2014}, the current formulation of Theorem \ref{intersect
rigidity} is inspired by the evaluation of hyperbolic mass
integrals on exhausting polytopes in the Poincar\'{e} half-space model
of an asymptotically hyperbolic manifold (see
\cite{jang-hyperbolic-2021} and \cite[Appendix B]{chai-dihedral-2024}).
Li {\cite{li-dihedral-2020-1}} showed the dihedral rigidity conjecture
for parabolic cubes up to dimension seven using free boundary constant
mean curvature hypersurfaces. Using capillary constant mean curvature
surfaces, the first named author and G. Wang
{\cite{chai-dihedral-2024}} proved three-dimensional dihedral rigidity
for certain prisms similar to the ones considered in
{\cite{li-polyhedron-2020}} and tetrahedra with a base face or a top
face which generalizes Li's approach
{\cite{li-dihedral-2020-1}}. Using spacetime harmonic functions, Tsang
{\cite{tsang-dihedral-2021-arxiv}} studied the dihedral rigidity for
three-dimensional cubical initial data sets which include parabolic
cubes. Wang-Xie {\cite{wang-dihedral-2023-arxiv}} used spinor methods
for polyhedral domains in hyperbolic space which are radially convex
and has a top face, in particular, for parabolic cubes.


Our proof uses the smoothing procedure of Brendle \cite{brendle-scalar-2024}, therefore shares the same restrictions of matching angles. Our new contribution to the problem is mainly the development of a new Dirac operator
and properties of the solutions of the related Dirac type equation (see Section \ref{sec:dihedral}).
Also, one should be able to obtain alternative versions of Theorem \ref{intersect rigidity} following a different smoothing procedure of \cite{brendle-gromovs-arxiv-2023} or using index theory on manifolds with corners in \cite{wang-gromovs-2022-arxiv}.

\

The article is organized as follows:

\

In Section \ref{sec:new dirac}, we consider a new twisted
spinor bundle and connection, and establish a corresponding integrated form of
the Schrodinger-Lichnerowicz formula. In Section
\ref{sec:llarull-boundary}, we prove the general version of Llarull
theorem (Theorem \ref{general Llarull}), and we also discuss 
Llarull theorem with boundary 
in the hyperbolic space in Subsection \ref{sec:llarull}. In Section \ref{sec:dihedral}, we focus on the proof of Theorem \ref{intersect rigidity}. In Section \ref{odd D}, we give a brief account of how to prove Theorems \ref{hyperbolic Llarull} and \ref{intersect rigidity} in odd dimensions.

\

\text{{\bfseries{Acknowledgements}}}. X. Chai was supported by the National Research Foundation of Korea (NRF) grant
funded by the Korea government (MSIT) (No. RS-2024-00337418) and an NRF grant No. 2022R1C1C1013511.
X. Wan was supported by the National Natural Science Foundation of China (Grant No. 12101093) and the Natural Science Foundation of Chongqing (Grant No. CSTB2022NSCQ-JQX0008), the Scientific Research Foundation of the Chongqing University of Technology.

\section{A new Dirac operator}\label{sec:new dirac}

In this section, we introduce a new connection \eqref{modified conn} and a new Dirac operator \eqref{modified dirac}. We compute the integral form of the related Schrodinger-Lichnerowicz formula in Proposition \ref{Dirac Witten}.

Consider $(X,h)$, an $(n-1)$-dimensional Riemannian manifold, and $X \times I$ equipped with a warped product metric $g_{0}$ of the form
\begin{equation}
  {g_{0}} = dr^2 + \varphi(r)^2 h,
\end{equation}
where $r$ is the parameter for the interval $I = [a,b]$ and $\varphi$ is a positive smooth function on $I$. Let $M$ be an $n$-dimensional compact manifold in $X \times I$.

Let $(N,{g})$ be a compact Riemannian manifold and $f: N \to M$ a spin map, meaning $TN \oplus f^{\ast}TM$ admits a spin structure. Denote the associated spinor bundle over $N$ by $S_N \otimes f^{\ast}S_M := S_{TN \oplus f^{\ast}TM}$, where $S_N$ and $S_M$ are the local spinor bundles over $N$ and $M$, respectively. Assuming $N$ and $M$ have the same dimension, the bundle $TN \oplus f^{\ast}TM$ has an even rank, and its spinor bundle carries a natural $\mathbb{Z}_2$-grading. If $\dim M$ is even, the $\mathbb{Z}_2$-grading of $S_N \otimes f^{\ast}S_M$ is given by ${\epsilon} \otimes \b\epsilon$, where ${\epsilon}$ and $\b\epsilon$ are the grading operators of $S_N$ and $f^{\ast}S_M$, respectively. For odd $\dim M$, computations are analogous to the even-dimensional case. Thus, we assume $\dim M$ is even.

We define
\begin{equation}\label{eqn-s}
  s(r) = \int_a^r \frac{1}{\varphi(t)}dt, \quad r \in [a,b]
\end{equation}
and 
\begin{equation}
  \psi(s) = \varphi(r(s)),\quad s \in I_s = \left[0, \int_a^b \frac{1}{\varphi(t)} dt \right].
\end{equation}
Then 
\begin{equation}
  {g_{0}} = dr^2 + \varphi(r)^2 h = \psi(s)^2 (ds^2 + h) = \psi(s)^2 \bar{g},
\end{equation}
where $\bar{g} := ds^2 + h$ is a direct product metric on $X \times I_s$. Recall that $M$ is a manifold in $X \times I$, which can also be viewed as a manifold in $X \times I_s$ by the bijective map \eqref{eqn-s}. Let $\nabla^{M}$ be the Levi-Civita connection on $(M,\bar{g})$, and $\nabla^{N}$ the Levi-Civita connection of $(N,{g})$. We also use only $\nabla$ to denote the Levi-Civita connection of $(N,g)$ for brevity
if there was no confusion. We denote the associated connections on $S_M$ and $S_N$ by $\nabla^{S_M}$ and $\nabla^{S_N}$, respectively. Then
\begin{equation}
  \nabla = \nabla^{S_N} \otimes 1 + 1 \otimes f^{\ast} \nabla^{S_M}
\end{equation}
is the natural connection on the bundle $S_N \otimes f^{\ast}S_M$. We denote the Clifford multiplication of a vector ${v} \in TN$ by ${c}({v})$ and the Clifford multiplication of a vector $\bar{v} \in f^{\ast}TM$ by $\bar{c}(\bar{v})$. The Dirac operator $D$ associated with $\nabla$ is given by
\begin{equation}
  D = \sum_{i=1}^n {c}({e}_i)\nabla_{{e}_i},
\end{equation}
where $\{{e}_i\}_{1 \leq i \le n}$ is the local orthonormal frame of $(TN,{g})$. For any $\bar{e}_i \in TM$, $\nabla^{M}_{\bar{e}_i} \partial_s = 0$, and thus $\nabla^{S_M}_{\bar{e}_i} c(\partial_s) = c(\partial_s) \nabla^{S_M}_{\bar{e}_i}$. Let ${e}_n$ (resp. $\bar{e}_n$) be the unit inner normal vector field of $\partial N$ (resp. $\partial M$). We define the boundary Clifford actions by
\begin{equation}
  \b{c}_\partial(\b{e}_\lambda) = \b{c}(\b{e}_n) \b{c}(\b{e}_\lambda),\quad c_\partial(e_\lambda) = c(e_n) c(e_\lambda)
\end{equation}
for $\b{e}_\lambda \in TM$ and $e_\lambda \in TN$. The boundary connections are defined by
\begin{equation}
  \nabla^{S_N,\partial}_{{e}_j} = \nabla^{S_N}_{{e}_j} + \frac{1}{2} {c}(\nabla^{N}_{{e}_j} {e}_n) {c}({e}_n)
\end{equation}
\begin{equation}
  \nabla^{S_M,\partial}_{\bar{e}_j} = \nabla^{S_M}_{\bar{e}_j} + \frac{1}{2} \bar{c}(\nabla^{M}_{\bar{e}_j} \bar{e}_n) \bar{c}(\bar{e}_n)
\end{equation}
and 
\begin{equation}
  \nabla^\partial = \nabla^{S_N,\partial} \otimes 1 + 1 \otimes f^{\ast}(\nabla^{S_M,\partial}).
\end{equation}
The boundary Dirac operator $D^\partial$ is defined by
\begin{equation}
  D^\partial = \sum_{j=1}^{n-1} {c}_\partial({e}_j) \nabla^\partial_{{e}_j}. \label{boundary dirac operator}
\end{equation}

We define a modified connection $\hat{\nabla}$ on $S_N \otimes f^{\ast} S_M$
by
\begin{equation}
  \hat{\nabla}_{e_i} \sigma = \nabla_{e_i} \sigma + \tfrac{1}{2} f^{\ast}
  \left( \tfrac{\psi'}{\psi^2} \right) (\epsilon \otimes \bar{\epsilon})
  \cdot (c (e_i) \otimes \bar{c} (\partial_s)) \label{modified conn}
\end{equation}
and its associated Dirac operator
\begin{equation}
  \hat{D} = c (e_i) \hat{\nabla}_{e_i} = D + \Psi \label{modified dirac}
\end{equation}
where
\[ \Psi = \tfrac{n}{2} f^{\ast} \left( \tfrac{\psi'}{\psi^2} \right)
   (\epsilon \otimes \bar{\epsilon}) \bar{c} (\partial_s), \]
where it should be clear that we write $\bar{c} (\partial_s)$ as $1 \otimes
\bar{c} (\partial_s)$ for simplicity. We also introduce a bundle map $\chi$ on sections of $S_N \otimes f^{\ast}
S_M$ restricted on the boundary
\begin{equation}
  \chi \sigma = (\epsilon \otimes \bar{\epsilon}) (c(e_n) \otimes
  \bar{c}(\bar{e}_n)) \sigma. \label{chi with vol}
\end{equation}
We recall the following local boundary condition $B$ from \cite[Definition 3.1]{wang-gromovs-2022-arxiv}. 
\begin{defn} \label{condition B}
A section $\sigma$ of $S_N \otimes f^{\ast} S_M$ over $N$ is said to satisfy the local boundary condition $B$ if 
\begin{equation}
  ({\epsilon} \otimes \bar{\epsilon})({c}({e}_n) \otimes \bar{c}(\bar{e}_n)) \sigma = -\sigma
\end{equation}
on $\partial N$. 
\end{defn}

\subsection{Schrodinger-Lichnerowicz formula}
We present the integrated Schrodinger-Lichnerowicz formula of \eqref{modified conn}.
\begin{proposition} \label{Dirac Witten}
  Let $\sigma \in S_N \otimes f^{\ast} S_M$, then
  \begin{equation}\label{boundary dirac chi anti}
    D^{\partial} \chi + \chi D^{\partial} = 0,
  \end{equation}
  (that is, $D^{\partial}$ and $\chi$ anti-commute) and
\begin{align}\label{eqn-D}
& \int_N | \hat{D} \sigma |^2 \\
  = & \int_N | \hat{\nabla} \sigma |^2  
      + \int_N \langle \mathcal{R} \sigma, \sigma \rangle  \\
    & \quad+\int_{N}  \tfrac{n - 1}{2} \langle c (\nabla  (f^{\ast} (\tfrac{\psi'}{\psi^2})) (\epsilon \otimes \bar{\epsilon}) \bar{c} (\partial_s) \sigma, \sigma \rangle + \tfrac{n (n - 1)}{4} f^{\ast} (\tfrac{\psi'}{\psi^2})^2 | \sigma |^2  \\
      & \quad + \int_{\partial N}
     \tfrac{1}{4} \langle D^{\partial} (\sigma + \chi \sigma), \sigma - \chi
\sigma \rangle + \tfrac{1}{4} \langle D^{\partial} (\sigma - \chi \sigma),
      \sigma + \chi \sigma \rangle \\
     &  \quad+  \int_{\partial N} \langle \mathcal{A} \sigma,\sigma \rangle + \tfrac{n - 1}{n}  \langle c (e_n) \Psi \sigma, \sigma \rangle,
 \end{align}
where 
\begin{equation}
\mathcal{R} = \frac{R_{{g}}}{4} - \frac{1}{2} \sum_{i,j} \left\langle \bar{R} f_* {w}_j, \bar{w}_i \right\rangle {c}({w}_j) \otimes \bar{c}(\bar{w}_i), \label{R}
\end{equation}
and
\begin{equation}\label{A}
  \mathcal{A}:= \tfrac{1}{2} H_g - \tfrac{1}{2} \sum_{1 \leq i \leq n-1} c (e_n) c
  (e_i) \bar{c} (\nabla^{T M}_{f_{\ast} e_i} \bar{e}_n) c (\bar{e}_n).
  \end{equation}
  Here, $\{e_i\}$ (resp. $\{\bar{e}_i\}$) is a local orthonormal frame of $T N$, $w_i \in \wedge^2 T N $ (resp. $\bar{w}_i \in \wedge^2 TM$) are of the form $e_j\wedge e_{k}$ (resp. $\bar{e}_j \wedge \bar{e}_k$), and $\bar{R}$ is the curvature operator of $(M,\bar{g})$.
\end{proposition}

\begin{proof}
The proof of that $D^{\partial}$ and $\chi$ anti-commute is a tedious but direct
calculation by using only (anti-)commutative properties of Clifford
multiplication and \ the definition of $D^{\partial}$ given in \eqref{boundary dirac operator}.
For the proof of \eqref{eqn-D}, first, we have
\begin{align}
& | \hat{D} \sigma |^2 \\
= & |D \sigma |^2 + \langle \Psi \sigma, D \sigma \rangle + \langle D
\sigma, \Psi \sigma \rangle + \langle \Psi \sigma, \Psi \sigma \rangle
\\
= & |D \sigma |^2 + \langle \Psi \sigma, D \sigma \rangle + \langle D
\sigma, \Psi \sigma \rangle + \tfrac{n^2}{4} f^{\ast}
(\tfrac{\psi'}{\psi^2})^2 | \sigma |^2 .
\end{align}
  By integration by parts, we have
\begin{align}
& \int_N |D \sigma |^2 \\
= & \int_N \langle D^2 \sigma, \sigma \rangle + \int_{\partial N} \langle
c (e_n) D \sigma, \sigma \rangle \\
= & \int_N \langle \nabla^{\ast} \nabla \sigma, \sigma \rangle + \langle
\mathcal{R} \sigma, \sigma \rangle + \int_{\partial N} \langle c (e_n) D
\sigma, \sigma \rangle,
\end{align}
where in the last equality we have used the Schrodinger-Lichnerowicz formula
\cite{gromov-positive-1983,lawson-spin-1989} on a twisted spinor bundle. By integration by parts
  on the first term, we see
  \[ \int_N |D \sigma |^2 = \int_N | \nabla \sigma |^2 + \langle \mathcal{R}
     \sigma, \sigma \rangle + \int_{\partial N} \langle \nabla_{e_n} \sigma,
     \sigma \rangle + \langle c (e_n) D \sigma, \sigma \rangle . \]
  We replace $| \nabla \sigma |^2$ by $| \hat{\nabla} \sigma |^2$ using
\begin{align}
& | \nabla \sigma |^2 \\
= & \left| \hat{\nabla}_{e_i} \sigma - \tfrac{1}{2} f^{\ast} \left(
\tfrac{\psi'}{\psi^2} \right) (\epsilon \otimes \bar{\epsilon})
\cdot (c (e_i) \otimes \bar{c} (\partial_s)) \sigma \right|^2 \\
= & | \hat{\nabla} \sigma |^2 - \langle \hat{\nabla}_{e_i} \sigma,
\tfrac{1}{2} f^{\ast} \left( \tfrac{\psi'}{\psi^2} \right) (\epsilon
\otimes \bar{\epsilon}) \cdot (c (e_i) \otimes \bar{c} (\partial_s))
\sigma \rangle \\
& \quad - \langle \tfrac{1}{2} f^{\ast} \left( \tfrac{\psi'}{\psi^2}
\right) (\epsilon \otimes \bar{\epsilon}) \cdot (c (e_i) \otimes
\bar{c} (\partial_s)) \sigma, \hat{\nabla}_{e_i} \sigma \rangle +
\tfrac{n}{4} f^{\ast} (\tfrac{\psi'}{\psi^2})^2 | \sigma |^2 \\
= & | \hat{\nabla} \sigma |^2 - \langle c (e_i) \hat{\nabla}_{e_i} \sigma,
\tfrac{1}{2} f^{\ast} \left( \tfrac{\psi'}{\psi^2} \right) (\epsilon
\otimes \bar{\epsilon}) \cdot (1 \otimes \bar{c} (\partial_s)) \sigma
\rangle \\
& \quad - \langle \tfrac{1}{2} f^{\ast} \left( \tfrac{\psi'}{\psi^2}
\right) (\epsilon \otimes \bar{\epsilon}) \cdot (1 \otimes \bar{c}
(\partial_s)) \sigma, c (e_i) \hat{\nabla}_{e_i} \sigma \rangle +
\tfrac{n}{4} f^{\ast} (\tfrac{\psi'}{\psi^2})^2 | \sigma |^2 \\
= & | \hat{\nabla} \sigma |^2 - \tfrac{1}{n} \langle \Psi \sigma, (D +
\Psi) \sigma \rangle - \tfrac{1}{n} \langle (D + \Psi) \sigma, \Psi \sigma
\rangle + \tfrac{n}{4} f^{\ast} (\tfrac{\psi'}{\psi^2})^2 | \sigma |^2
\\
= & | \hat{\nabla} \sigma |^2 - \tfrac{1}{n} \langle \Psi \sigma, D \sigma
\rangle - \tfrac{1}{n} \langle D \sigma, \Psi \sigma \rangle -
\tfrac{n}{4} f^{\ast} (\tfrac{\psi'}{\psi^2})^2 | \sigma |^2 .
\end{align}
  To collect all the calculations in the above, we see
\begin{align}
& \int_N | \hat{D} \sigma |^2 \\
= & \int_N | \hat{\nabla} \sigma |^2 + \int_{\partial N} \langle
\nabla_{e_n} \sigma, \sigma \rangle + \langle c (e_n) D \sigma, \sigma
\rangle . \\
& \quad + \int_N \langle \mathcal{R} \sigma, \sigma \rangle + \tfrac{n -
1}{n} \langle \Psi \sigma, D \sigma \rangle + \tfrac{n - 1}{n} \langle D
\sigma, \Psi \sigma \rangle + \tfrac{n (n - 1)}{4} f^{\ast}
(\tfrac{\psi'}{\psi^2})^2 | \sigma |^2 . \label{with Psi}
\end{align}
  Now we handle the term containing $\Psi \sigma$. Assume that at a
  point $p \in N$, $\{e_i \}$ is a geodesic normal frame, that is,
  $\nabla_{e_i} e_j$ vanishes at $p$ for all $i$ and $j$. Also at this point,
\begin{align}
& \tfrac{2}{n} \nabla_{e_i} \langle \Psi \sigma, c (e_i) \sigma \rangle
\\
= & \langle (\nabla_{e_i} \Psi) \sigma, c (e_i) \sigma \rangle + \langle
\Psi \nabla_{e_i} \sigma, c (e_i) \sigma \rangle + \langle \Psi \sigma, c
(e_i) \nabla_{e_i} \sigma \rangle \\
= & \langle \nabla_{e_i} (f^{\ast} (\tfrac{\psi'}{\psi^2}) (\epsilon
\otimes \bar{\epsilon}) \bar{c} (\partial_s) \sigma, c (e_i) \sigma
\rangle + \langle \Psi c (e_i) \nabla_{e_i} \sigma, \sigma \rangle +
\langle \Psi \sigma, D \sigma \rangle \\
= & - \langle c (\nabla  (f^{\ast} (\tfrac{\psi'}{\psi^2})) (\epsilon
\otimes \bar{\epsilon}) \bar{c} (\partial_s) \sigma, c (e_i) \sigma
\rangle + \langle \Psi D \sigma, \sigma \rangle + \langle \Psi \sigma, D
\sigma \rangle \\
= & - \langle c (\nabla  (f^{\ast} (\tfrac{\psi'}{\psi^2})) (\epsilon
\otimes \bar{\epsilon}) \bar{c} (\partial_s) \sigma, c (e_i) \sigma
\rangle + \langle D \sigma, \Psi \sigma \rangle + \langle \Psi \sigma, D
\sigma \rangle .
\end{align}
  Hence, by integration by parts, we have
\begin{align}
& \int_N [\langle D \sigma, \Psi \sigma \rangle + \langle \Psi \sigma, D
\sigma \rangle] \\
= & \int_N \tfrac{n}{2} \langle c (\nabla  (f^{\ast}
(\tfrac{\psi'}{\psi^2})) (\epsilon \otimes \bar{\epsilon}) \bar{c}
(\partial_s) \sigma, \sigma \rangle + \int_{\partial N} \tfrac{n}{2}
\langle c (e_n) \Psi \sigma, \sigma \rangle.
\end{align}
With this in \eqref{with Psi}, it remains to deal with the boundary term to finish
the proof of the proposition.
It follows from the definition of $D^{\partial}$ that
\[ \nabla_{e_n} + c (e_n) D = D^{\partial} +\mathcal{A} .\]
Since $D^{\partial}$ and $\chi$ anti-commute, 
\begin{align}
  & \langle D^{\partial} \sigma, \sigma \rangle \\
  = & \tfrac{1}{4} \langle D^{\partial} ((1 + \chi) \sigma + (1 - \chi)
  \sigma), (1 + \chi) \sigma + (1 - \chi) \sigma \rangle \\
  = & \tfrac{1}{4} \langle D^{\partial} (\sigma + \chi \sigma), \sigma - \chi
  \sigma \rangle + \tfrac{1}{4} \langle D^{\partial} (\sigma - \chi \sigma),
  \sigma + \chi \sigma \rangle . 
\end{align}
And the proof is complete.
\end{proof}

\subsection{Schrondinger-Lichnerowicz formula with comparisons}

Now we establish a consequence of Proposition \ref{Dirac Witten} when the conditions of Theorem \ref{general Llarull}
are in effect
and the section $\sigma\in S_N \otimes f^{\ast}S_M$ satisfies the local boundary condition $B$ (Definition \ref{condition B}).
More specifically, we have the following.

\begin{prop}\label{sl formula}
Let $M$ be an $n$-dimensional compact manifold in the warped product manifold $(X \times I, {g}_{0} = dr^2 + \varphi(r)^2 h) = (X \times I_s, \psi(s)^2 \bar{g})$ such that the boundary $\partial M$ is convex with respect to $\bar{g}$ and $(\log \varphi)'' \leq 0$. If the curvature operator of $(X, h)$ is non-negative and $f: (N, {g}) \to (M, {g}_{0})$ is distance non-increasing, then
\begin{equation}\label{eqn-5}
 \int_N | \hat{D} \sigma |^2  \geqslant  \int_N | \hat{\nabla} \sigma |^2  
     + \tfrac{1}{4}\int_N (R_g - f^{\ast} R_{g_0}) |\sigma|^2
       + \tfrac{1}{2} \int_{\partial N} (H_{{g}} - f^{\ast} H_{{g}_{0}}) |\sigma|^2 
\end{equation}
for every smooth section $\sigma$ of $S_N \otimes f^{\ast} S_M$ that satisfies the local boundary condition B (Definition \ref{condition B}).
\end{prop}

Before we proceed, we recall a variant of \cite[Lemmas 2.8, 2.9]{wang-dihedral-2023-arxiv}. Earlier inequalities in the same spirit of \eqref{listing} can be found in \cite{listing-scalar-arxiv-2010}, \cite{goette-scalar-2002}. To obtain Lemma \ref{lemma1}, a simple scaling of \cite{wang-dihedral-2023-arxiv} would suffice, however, we need some parts of the proof later, so we repeat the proof.
\begin{lemma}\label{lemma1}
  Let $\mathcal{R}$ and $\mathcal{A}$ be given in \eqref{R} and \eqref{A}, we have the following.
  \begin{enumerate}
  \item[(a)] If the curvature operator on $(X,h)$ is non-negative, then 
    \begin{equation} \label{listing}
      \mc{R} \geq \frac{R_{{g}}}{4} - f^{\ast}\left(\frac{R_h}{4 \psi^2}\right),
    \end{equation}
    where $R_h$ denotes the scalar curvature of $(X,h)$. 
  \item[(b)] If the second fundamental form $A$ of $\partial M \subset (M,\bar{g})$ is non-negative, then 
    \begin{equation}\label{eqn-8}
      \mc{A} \geq \frac{H_{{g}}}{2} - f^{\ast}\left(\frac{H_{\bar{g}}}{2 \psi}\right),
    \end{equation}
    where $H_{{g}}$ (resp. $H_{\bar{g}}$) denotes the mean curvature of $\partial N \subset (N,{g})$ (resp. $\partial M \subset (M,\bar{g})$).
  \end{enumerate}
\end{lemma}
\begin{proof}
  The proofs of (a) and (b) are similar, so we only provide the proof of (a).
  Recall from \eqref{R} that
  \begin{equation}
\mathcal{R} = \frac{R_{{g}}}{4} - \frac{1}{2} \sum_{i,j} \left\langle \bar{R} f_* {w}_j, \bar{w}_i \right\rangle {c}({w}_j) \otimes \bar{c}(\bar{w}_i).
\end{equation}
Given that the curvature operator $\bar{R}$ is non-negative along each leaf, there exists a self-adjoint operator $\bar{L} = \operatorname{End}(\wedge^2 TX)$ such that $\bar{R} = \bar{L}^2$, meaning $\left\langle \bar{R} \bar{w}_j, \bar{w}_i \right\rangle_M = \left\langle \bar{L} \bar{w}_j, \bar{L} \bar{w}_i \right\rangle_M$. Set
\begin{equation}
{L} \bar{w}_k = \sum_i \left\langle \bar{L} \bar{w}_k, f_* {w}_i \right\rangle_M {w}_i \in \wedge^2 TN.
\end{equation}
Then, we have
\begin{align}
\begin{split}
&\quad  -\frac{1}{2} \sum_{i,j} \left\langle \bar{R} f_* {w}_j, \bar{w}_i \right\rangle_M {c}({w}_j) \otimes \bar{c}(\bar{w}_i) \\
&= -\frac{1}{2} \sum_{i,j,k} \left\langle \bar{L} (f^{\ast} {w}_j), \bar{w}_k \right\rangle_M \left\langle \bar{L} \bar{w}_i, \bar{w}_k \right\rangle_M {c}({w}_j) \otimes \bar{c}(\bar{w}_i) \\
&= -\frac{1}{2} \sum_k {c}(f^{\ast}\psi {L} \bar{w}_k) \otimes \bar{c}((f^{\ast}\psi)^{-1} \bar{L} \bar{w}_k) \\
&= \frac{1}{4} \sum_k \left[ {c}(f^{\ast}\psi {L} \bar{w}_k)^2 \otimes 1 + 1 \otimes \bar{c}((f^{\ast}\psi)^{-1} \bar{L} \bar{w}_k)^2\right.\\
&\quad - \left.({c}(f^{\ast}\psi {L} \bar{w}_k) \otimes 1 + 1 \otimes \bar{c}((f^{\ast}\psi)^{-1} \bar{L} \bar{w}_k))^2 \right] \\
&\geq \frac{1}{4} \sum_k ({c}(f^{\ast}\psi {L} \bar{w}_k)^2 \otimes 1 + 1 \otimes \bar{c}((f^{\ast}\psi)^{-1} \bar{L} \bar{w}_k)^2),
\end{split}
\end{align}
where the last inequality follows from the fact that the element
\begin{equation}
{c}(f^{\ast}\psi {L} \bar{w}_k) \otimes 1 + 1 \otimes \bar{c}((f^{\ast}\psi)^{-1} \bar{L} \bar{w}_k)
\end{equation}
is skew-symmetric, hence its square is non-positive.

Applying the same reasoning as the Lichnerowicz formula, we find
\begin{equation}
\sum_k \bar{c}((f^{\ast}\psi)^{-1} \bar{L} \bar{w}_k)^2 = -f^{\ast}(\tfrac{R_h}{2 \psi^2}).
\end{equation}

Similarly, by the definition of ${L}$, we have
\begin{align}
\begin{split}
\sum_k {c}(f^{\ast}\psi {L} \bar{w}_k)^2 &= (f^{\ast}\psi)^2 \sum_{i,j,k} \left\langle {L} \bar{w}_k, f_* {w}_i \right\rangle \left\langle {L} \bar{w}_k, f_* {w}_j \right\rangle {c}({w}_i) {c}({w}_j) \\
&= (f^{\ast}\psi)^2 \sum_{i,j} \left\langle \bar{R}(f_* {w}_i), f_*{w}_j \right\rangle {c}({w}_i){c}({w}_j).
\end{split}
\end{align}
We choose a local ${g}$-orthonormal frame ${e}_1, \ldots, {e}_n$ of $TN$ and a local $g_0$-orthonormal frame $\bar{e}_1, \ldots, \bar{e}_n$ of $TM$ such that $f_* {e}_i = \mu_i \bar{e}_i$ with $\mu_i \geq 0$. Then $f_{*}({e}_i \wedge {e}_j) = \mu_i \mu_j \bar{e}_i \wedge \bar{e}_j$. Since $f: (N, {g}) \to (M, \psi^2 \bar{g})$ is distance non-increasing, $\mu_i \psi \leq 1$. Hence
\begin{equation}\label{eqn-7}
\sum_k {c}(f^{\ast}\psi {L} \bar{w}_k)^2 = -(f^{\ast}\psi)^2 \sum_{i,j} \mu_i^2 \mu_j^2 (f^{\ast}\bar{R}_{ijji}) \geq -f^{\ast}(\frac{R_h}{2 \psi^2}).
\end{equation}
The proof is complete.
\end{proof}

Now we can finish the proof of Proposition \ref{sl formula}.

\begin{proof}[Proof of Proposition \ref{sl formula}]
Because $(\psi' / \psi^2)' \leq 0$ (equivalent to $(\log \varphi)'' \leq 0$) and $f: (N, {g}) \to (M, {g}_{0})$ is distance non-increasing, so 
\begin{equation}\label{eqn-1}
  \left\langle {c}({\nabla}(f^{\ast} \frac{n \psi'}{2 \psi^2})) ({\epsilon} \otimes \bar{\epsilon }\bar{c}(\partial_s)) \sigma, \sigma \right\rangle \geq f^{\ast}\left((\frac{n \psi'}{2 \psi^2})'\right) \frac{1}{f^{\ast} \psi} |\sigma|^2.
\end{equation}
Note that $g_0$ is a warped product, and the scalar curvature $R_{g_0}$ is given by
\begin{align}
 \frac{R_{g_0}}{4} &=  \frac{R_h}{4 \varphi^2} - \frac{n(n-1)}{4} \left(\frac{ \varphi'}{ \varphi}\right)^2 - \frac{n-1}{2} \frac{d}{dr} \left(\frac{ \varphi'}{ \varphi}\right)\\
  &= \frac{R_h}{4 \psi^2} - \frac{n(n-1)}{4} \left(\frac{ \psi'}{ \psi^2}\right)^2 - \frac{n-1}{2\psi(s)} \frac{d}{ds} \left(\frac{n \psi'}{2 \psi^2}\right).
\end{align}
Using Lemma \ref{lemma1} (a) and \eqref{eqn-1}, we obtain
\begin{align}\label{eqn-R}
\begin{split}
 &\quad  \left\langle \mc{R} \sigma, \sigma \right\rangle + \frac{n-1}{2}\left\langle {c}({\nabla}(f^{\ast} \frac{ \psi'}{ \psi^2})) ({\epsilon} \otimes \bar{ \epsilon }\bar{c}(\partial_s)) \sigma, \sigma \right\rangle + \frac{n(n-1)}{4} (f^{\ast} \frac{ \psi'}{ \psi^2})^2 |\sigma|^2 \\
&  \geq \tfrac{1}{4} (R_{{g}} - f^{\ast} R_{g_{0}}) |\sigma|^2.
\end{split}
\end{align}
As the second fundamental form $A$ of $\partial M \subset (M, \bar{g})$ is non-negative, then by \eqref{eqn-8},
\begin{align}\label{eqn-2}
\begin{split}
 \langle \mc{A} \sigma, \sigma \rangle \geq \left(\frac{H_{{g}}}{2} - f^{\ast}(\frac{H_{\bar{g}}}{2 \psi})\right) |\sigma|^2.
\end{split}
\end{align}
For the term $\left\langle {c}({e}_n) \Psi \sigma, \sigma \right\rangle$, first we note that for any $Y$ that is orthogonal to
$\bar{e}_n$, $(\epsilon c(e_n)) \otimes (\bar{\epsilon}\bar{c}(\bar{e}_n))$ and
$(\epsilon c(e_n)) \otimes (\bar{\epsilon}\bar{c}(Y))$ anti-commute, also because that $\sigma$ satisfies $\chi \sigma = -\sigma$, so
$\langle (\epsilon c(e_n)) \otimes (\bar{\epsilon}\bar{c}(Y))\sigma, \sigma \rangle = 0$, hence
\begin{align}\label{eqn-3}
\begin{split}
  \left\langle {c}({e}_n) \Psi \sigma, \sigma \right\rangle = \left\langle \bar{e}_n, \partial_s \right\rangle f^{\ast}(\frac{n \psi'}{2 \psi^2}) |\sigma|^2 = f^{\ast}(\frac{n \bar{e}_n(\psi)}{2 \psi^2})|\sigma|^2.
\end{split}
\end{align}
As $\bar{g}$ and $g_0$ are conformal, their mean curvatures are related by
\begin{equation}\label{eqn-4}
  H_{{g}_{0}} = \frac{1}{\psi} H_{\bar{g}} - (n-1) \frac{1}{\psi^2} \bar{e}_n(\psi).
\end{equation}
By \eqref{eqn-2}, \eqref{eqn-3}, and \eqref{eqn-4}, we obtain
\begin{equation}\label{eqn-H}
 \langle \mathcal{A} \sigma, \sigma \rangle + \tfrac{n - 1}{n} \langle c
   (e_n) \Psi \sigma, \sigma \rangle \geqslant (\tfrac{H_g}{2} -
   \tfrac{f^{\ast} H_{g_0}}{2}) | \sigma |^2 . 
\end{equation}
Using \eqref{eqn-D}, \eqref{eqn-R}, and \eqref{eqn-H}, we finish the proof of the proposition. 
\end{proof}

\section{Llarull theorems with boundary}\label{sec:llarull-boundary}

In this section, we establish Theorems \ref{general Llarull} and \ref{hyperbolic Llarull}
using the Dirac operator and Schrodinger-Lichnerowicz formula introduced in the previous section.

\subsection{Scalar curvature rigidity in a warped product}

Now we
establish the new rigidity result (Theorem \ref{general Llarull}) concerning scalar curvature and
mean curvature for domains in warped product manifolds.
The ingredients are the methods from \cite{wang-dihedral-2023-arxiv},
the Dirac operator introduced Section \ref{sec:new dirac} and the
Schrondinger-Lichnerowicz formula (Proposition \ref{sl formula}).

\begin{proof}[Proof of Theorem \ref{general Llarull}]
	Note that $\hat{D}$ only differs from the usual twisted Dirac operator on $S_N \otimes f^{\ast} S_M$ by a bounded endomorphism. Therefore, $\hat{D}$ with the local boundary condition $B$ is a Fredholm operator, and its Fredholm index is
\begin{equation}
  \mathrm{Ind}(\hat{D}) = \deg(f) \cdot \chi(M), \label{index theory}
\end{equation}
where $\deg(f)$ is the degree of $f$ and $\chi(M)$ is the Euler characteristic of $M$. By assumption, $\mathrm{Ind}(\hat{D}) \neq 0$. It follows that there exists a non-zero smooth section $\sigma$ of $S_N \otimes f^{\ast} S_M$ satisfying the local boundary condition $B$ such that $\hat{D} \sigma = 0$. Using \eqref{eqn-5} and conditions (i) and (ii) of Theorem \ref{general Llarull}, $\sigma$ satisfies the following equation
\begin{equation}\label{eqn-6}
  \nabla_{{\xi}} \sigma - \frac{1}{n} {c}({\xi}) \Psi \sigma = \nabla_{{\xi}} \sigma + \frac{1}{n} {c}({\xi}) D \sigma= 0
\end{equation}
for all ${\xi} \in TN$. For any given point $x \in N$ and any smooth path in $N$ starting at $x$, the restriction of $\sigma$ to the path satisfies the homogeneous differential equation given by \eqref{eqn-6}. By the uniqueness of the solution, $\sigma$ vanishes along the entire path if $\sigma(x) = 0$. Without loss of generality, assume $N$ is connected. Since $\sigma$ is a non-zero section, the above discussion implies that $\sigma$ is non-zero everywhere on $N$. Hence, the inequality \eqref{eqn-5} implies that the equalities in (i) and (ii) of Theorem \ref{general Llarull} hold. 

If $\varphi$ is strictly log-concave, then the equality in \eqref{eqn-1} holds, meaning
\begin{equation}
  ({\epsilon} \otimes \bar{\epsilon}) ({c}({\nabla} {(f^{*}s)}) \otimes \bar{c}(\partial_s)) \sigma = -\frac{1}{f^{\ast} \psi} \sigma,
\end{equation}
 Then
\begin{equation}
  {c}(f^{\ast} \psi  \nabla(f^{*}{s})) \sigma = -({\epsilon} \otimes \bar{\epsilon}) \bar{c}(\partial_s) \sigma.
\end{equation}
Since $\sigma$ is non-zero everywhere, we have $| \nabla (f^{*}{s})| = \frac{1}{f^{\ast} \psi}$. From \eqref{eqn-6}, we have
\begin{align}
\begin{split}
  \nabla_{{\xi}} \sigma &= \frac{1}{n} {c}({\xi}) \Psi \sigma \\
  &= \frac{1}{2} f^{\ast}(\frac{\psi'}{\psi^2}) {c}({\xi}) ({\epsilon} \otimes \bar{\epsilon}) \bar{c}(\partial_s) \sigma \\
  &= -\frac{1}{2} f^{\ast}(\frac{\psi'}{\psi}) {c}({\xi}) {c}(\nabla(f^{*} {s})) \sigma.
 \end{split}
\end{align}
Thus
\begin{align}
\begin{split}
  {c}(\nabla_{{\xi}}(f^{\ast} \psi \nabla (f^{*}s) )) \sigma &= \nabla_{{\xi}}(c(f^{\ast} \psi  \nabla (f^{*}s) ) \sigma) - c(f^{\ast} \psi  \nabla (f^{*}s)  ) \nabla_{{\xi}} \sigma \\
  &= -({\epsilon} \otimes \bar{\epsilon}) \bar{c}(\partial_s) \nabla_{{\xi}} \sigma - c(f^{\ast} \psi  \nabla (f^{*}s)  ) \nabla_{{\xi}} \sigma \\
  &= \frac{1}{2} f^{\ast}(\frac{\psi'}{\psi}) ({\epsilon} \otimes \bar{\epsilon}) \bar{c}(\partial_s) {c}({\xi}) {c}( \nabla (f^{*}s) ) \sigma \\
  &\quad + \frac{1}{2} f^{\ast}(\frac{\psi'}{\psi}) c(f^{\ast} \psi \nabla(f^{*}s) ) {c}({\xi}) {c}(\nabla(f^{*}s)) \sigma \\
  &= \frac{1}{2} f^{\ast}(\frac{\psi'}{\psi^2}) {c}({\xi}) \sigma + \frac{1}{2} f^{\ast}(\frac{\psi'}{\psi^2}) {c}({\xi}) \sigma \\
  &\quad - f^{\ast}(\psi') \left\langle {\xi}, \nabla(f^{*}s) \right\rangle {c}( \nabla(f^{*}s) ) \sigma \\
  &= {c} \left(f^{\ast}(\frac{\psi'}{\psi^2}) ({\xi} - \left\langle {\xi}, f^{\ast} \psi { \nabla(f^{*}s) } {s} \right\rangle f^{\ast} \psi { \nabla(f^{*}s) }) \right) \sigma,
 \end{split}
\end{align}
which implies that
\begin{equation}
  \nabla_{{\xi}}(f^{\ast} \psi \nabla (f^{*}s)) = f^{\ast}(\frac{\psi'}{\psi^2}) ({\xi} - \left\langle {\xi}, f^{\ast} \psi  \nabla (f^{*}s)  \right\rangle f^{\ast} \psi  \nabla (f^{*}s)).
\end{equation}
For any ${\xi}, {\eta} \in TN$, we have
\begin{align}
\begin{split}
  (\nabla (f^{\ast}(\psi d s)))({\xi}, {\eta}) &= \left(\nabla_{{\xi}} f^{\ast}(\psi d s)\right)({\eta}) \\
  &= \left\langle \nabla_{{\xi}} (f^{\ast} \psi ( \nabla (f^{*} s))) {\eta} \right\rangle \\
  &= f^{\ast}(\frac{\psi'}{\psi^2}) \left({g} - f^{\ast}(\psi ds) \otimes f^{\ast}(\psi ds)\right)({\xi}, {\eta}).
 \end{split}
\end{align}
In particular, the flow lines generated by the vector field $f^{\ast} \psi \nabla (f^{*}s) $ on $N$ are unit speed geodesics. 
For any $x \in N \backslash \partial N$, we assume that
\begin{equation}
  f(x) \in (X \times \{s\}) \cap M
\end{equation}
for some $s \in I_s$. Then $x \in f^{-1}((X \times \{s\}) \cap M)$. Along the geodesic generated by $f^{\ast} \psi \nabla (f^{*} s) $ starting from $x$, the metric ${g}$ around $x$ has the form 
\begin{equation}
  {g} = f^{\ast}(\psi ds) \otimes f^{\ast}(\psi ds) + {g}_{{s}},
\end{equation}
where $g_{s}$ is the induced metric on the level set.
By a direct calculation, one finds 
\begin{equation}
  \mathcal{L}_{f^{\ast} \psi \nabla(f^{*}s) }({g}) = 2 \nabla (f^{\ast}(\psi d s)) = 2 f^{\ast}(\frac{\psi'}{\psi^2}) \left({g} - f^{\ast}(\psi ds) \otimes f^{\ast}(\psi ds)\right),
\end{equation}
where $\mathcal{L}$ denotes the Lie derivative. Thus 
\begin{equation}
  \mathcal{L}_{f^{\ast} \psi \nabla (f^{*}s) }({g}_{{s}}) = 2 f^{\ast}(\frac{\psi'}{\psi^2}) {g}_{{s}}.
\end{equation}
We denote ${g}_{{s}} := (f^{\ast}\psi)^2 {\hat{h}}$, then
\begin{align}\label{eqn-9}
\begin{split}
  0 &= \mathcal{L}_{f^{\ast} \psi \nabla (f^{*}s) }({g}_{{s}}) - 2 f^{\ast}(\frac{\psi'}{\psi^2}) {g}_{{s}} \\
  &= \frac{f^{\ast} \psi \nabla (f^{*}s) ((f^{\ast} \psi)^2)}{(f^{\ast} \psi)^2} {g}_{{s}} - 2 f^{\ast}(\frac{\psi'}{\psi^2}) {g}_{{s}} + (f^{\ast} \psi)^2 \mathcal{L}_{f^{\ast} \psi \nabla (f^{*}s) }({\hat{h}}) \\
  &= (f^{\ast} \psi)^2 \mathcal{L}_{f^{\ast} \psi \nabla (f^{*}s) }({\hat{h}}),
 \end{split}
\end{align}
where the last equality holds since
\begin{align}
\begin{split}
\nabla (f^{*}s) ((f^{\ast} \psi)^2) &= 2 (f^{\ast} \psi) (f^{\ast} \psi') \nabla(f^{*}s) (d( f^{\ast} s)) \\
  &= 2 (f^{\ast} \psi) (f^{\ast} \psi') |\nabla (f^{*}s)|^2 \\
  &= 2 f^{\ast}(\tfrac{\psi'}{\psi}).
 \end{split}
\end{align}
From \eqref{eqn-9}, we know that ${\hat{h}}$ is a metric on $f^{-1}((X \times \{s\}) \cap M)$ around $x$, and is independent of the direction $f^{\ast} \psi \nabla (f^{*}s) $. Thus the metric ${g}$ has the following form around $x$
\begin{equation}
  {g} = (f^{\ast} \psi)^2 (f^{\ast}(ds^2) + {\hat{h}}).
\end{equation}
Since the equality in \eqref{eqn-5} holds, the equality in \eqref{eqn-7} also holds, that is,
\begin{equation}
  \sum_{i \neq j} (1 - (\psi \mu_i) (\psi \mu_j)) R_{ijji} = 0,
\end{equation}
where $1 \leq i, j \leq n-1$ such that $\bar{e}_i, \bar{e}_j \in T((X \times \{t\}) \cap M)$ since $R$ is the leaf-wise curvature along each leaf $M \cap (X \times \{t\})$.
If $(X, h)$ has non-negative sectional curvature and positive Ricci curvature, then $R_{ijji} \geq 0$, and for any $i$, there exists $i_0 \neq i$ such that $R_{ii_0 i_0 i} > 0$. Note that $\psi \mu_i \leq 1$ for any $i$, so
\begin{equation}
  (1 - (\psi \mu_i)(\psi \mu_{i_0})) R_{ii_0 i_0 i} = 0,
\end{equation}
which implies that $\psi \mu_i = 1$ for any $1 \leq i \leq n-1$. Thus
\begin{align}
\begin{split}
  (f^{\ast}(\psi^2 h))({e}_i, {e}_j) &= (f^{\ast} g)({e}_i, {e}_j) \\
  &= g(f_* {e}_i, f_* {e}_j) \\
  &= \psi^2 g_0(\mu_i e_i, \mu_j e_j) \\
  &= \delta_{ij} = {g}({e}_i, {e}_j) \\
  &= (f^{\ast} \psi)^2 {\hat{h}}({e}_i, {e}_j).
 \end{split}
\end{align}
It follows that $f^{\ast} h = {\hat{h}}$. Thus
\begin{equation}
  {g} = (f^{\ast} \psi)^2 (f^{\ast}(ds^2) + f^{\ast} h) = f^{\ast} g_0,
\end{equation}
that is, $f: (N, {g}) \to (M, g_{0})$ is a local isometry. 

If $\partial M$ is strictly convex with respect to $\bar{g}$, then the second fundamental form is positive, that is, $A(\bar{e}_a, \bar{e}_a) > 0$ for any non-zero $\bar{e}_a \in T(\partial M)$. Similarly, the equality in \eqref{eqn-8} implies that
\begin{equation}
  f_* {e}_a = \tfrac{1}{\psi} \bar{e}_a, \quad 1 \leq a \leq n-1
\end{equation}
for an orthonormal basis $\{e_a\}_{1 \leq a \leq n-1}$ (resp. $\{e_a\}_{1 \leq a \leq n-1}$) of $(T(\partial N), {g})$ (resp. $(T(\partial M), {g}_0)$).
Hence $(f^{\ast} {g_{0}})({e}_a, {e}_b) = {g}_{0}({e}_a, {e}_b)$, that is, $\partial f: (\partial N, {g}) \to (\partial M, {g}_{0})$ is a local isometry.
\end{proof}

\subsection{Rigidity of smooth domains in Poinca\'{e} half-space model}\label{sec:llarull}

Now we prove the rigidity Theorem \ref{hyperbolic Llarull} for smooth domains that are strictly convex 
in $(\mb{R}^n_+,\delta)$ of a Poincar\'{e} half-space model of the hyperbolic space.
Here, we assume that $n$ is even. Note that $b=\psi(x^1)^2\delta$, where $\psi(x^1)=\frac{1}{x^1}$. Then 
$
  \frac{\psi'}{\psi^2}=-1
$,
and, \eqref{eqn-1} still holds without assuming that $f=\mathrm{id}: (\Omega, g) \to (\Omega, b)$ is distance non-increasing in the interior of $\Omega$. By Theorem \ref{general Llarull}, $R_g = -n(n-1)$, $H_g = H_b$, and $\left.g\right|_{\partial \Omega} = \left.b\right|_{\partial \Omega}$. It remains to show that $(\Omega, g)$ is hyperbolic.

Let $\sigma$ be a non-zero section of $S_{\Omega_g} \otimes S_{\Omega_\delta}$ satisfying the local boundary condition $B$ such that $\hat{D} \sigma = 0$, where $\Omega_\delta$ (resp. $\Omega_g$) denotes the Riemannian manifold $(\Omega, \delta)$ (resp. $(\Omega, g)$). Let $\{\bar{s}_\alpha\}_{1\leq \alpha\leq m}$, $m=2^{n/2}$, be an orthonormal basis of the space $\Delta_n$ of spinors, which is also a global parallel frame of the spinor bundle
$S_{\Omega_\delta}=\mb{R}^n_+\times \Delta_n$. Then 
\begin{equation}
  \sigma=\sum_{\alpha=1}^{m}{s}_\alpha\otimes \bar{s}_\alpha,
\end{equation}
where ${s}_\alpha$ are smooth sections of $S_{\Omega_g}$.  Denote by $s=(s_1,\ldots,s_m)$ the $m$-tuple of spinors.
 The $m$-tuple of spinors $s$ should be understood as a spinor-valued column vector. Given a unit Euclidean vector $X$, we define the matrix $\omega_X$ as
\begin{equation}
  (\omega_{X}{s})_\alpha=\sum_{\beta=1}^m\omega_{X\alpha\beta}{s}_\beta=\sum_{\beta=1}^m\left\langle \b{\epsilon} \b{c}(X)\bar{s}_\beta,\bar{s}_\alpha\right\rangle {s}_\beta,
\end{equation}
where 
\begin{equation}
    \omega_{X\alpha\beta}:=\left\langle \b{\epsilon}\b{c}(X)\bar{s}_\beta,\bar{s}_\alpha\right\rangle.
\end{equation}
Hence,
\begin{align}
\begin{split}
  0&=\nabla_{{\xi}}\sigma- \tfrac{1}{n} {c}({\xi}) \Psi \sigma\\
  &=\nabla_{{\xi}}\sigma- \tfrac{1}{2}({\epsilon} \otimes \bar{\epsilon}) \cdot ({c}({\xi})  \otimes \bar{c}(\partial/\partial x^1)) \sigma\\
  &=\sum_{\alpha}(\nabla_{{\xi}}{s}_\alpha)\otimes \bar{s}_\alpha-\tfrac{1}{2}({\epsilon} \otimes \bar{\epsilon}) \sum_{\beta,\alpha} {c}({\xi}){s}_\beta\otimes  \langle \bar{c}(\partial/\partial x^1)\bar{s}_\beta,\bar{s}_\alpha\rangle \bar{s}_\alpha\\
  &=\sum_{\alpha}(\nabla_{{\xi}}{s}_\alpha-\tfrac{1}{2}\epsilon{c}({\xi})(\omega_{N_0}{s})_\alpha)\otimes \bar{s}_\alpha
 \end{split}
\end{align}
where $N_0:=\tfrac{\partial}{\partial{x^1}}$.
Hence $s$ satisfies the following equation
\begin{equation}\label{eqn-10}
  \nabla_{{\xi}}{s}-\tfrac{1}{2}\epsilon{c}({\xi})\omega_{N_0}{s}=0. 
\end{equation}
We can choose basis of $S_{\Omega_{\delta}}$ such that the matrix $\omega_{N_0}$ is diagonal with entries $\pm 1$ which reduces \eqref{eqn-10} to
\begin{equation}
  \nabla_{{\xi}}{s}_\alpha \mp\tfrac{1}{2}\epsilon{c}({\xi}){s}_\alpha =0. 
\end{equation}
We would like to point out that the spinor ${s}_\alpha$, which satisfies the above, has already been studied in asymptotically hyperbolic spin manifold \cite{minoo-scalar-1989}.

Similarly, the boundary condition $({\epsilon}\otimes \bar{\epsilon})({c}(\nu)\otimes \bar{c}(N))\sigma=-\sigma$ is equivalent to 
\begin{equation}
  {\epsilon}{c}(\nu)\omega_N{s}=-{s}, \label{bc brendle}
\end{equation}
where $\nu$ (resp. $N$) is the inward-pointing unit normal vector  of $\partial \Omega$ with respect to $g$ (resp. $\delta$).
We introduce a \text{{\itshape{formal}}} inner product $\langle c, s \rangle$
of a spinor $c \in \mathbb{C}^m$ and $m$-tuple $s$ of spinors by
\begin{equation}
  \langle c, s \rangle = \sum_{\alpha = 1}^m \bar{c}_\alpha s_\alpha, \label{c s}
\end{equation}
where $\bar{c}_i$ is the complex conjugate of $c \in \mathbb{C}$. To avoid double
levels of angular brackets, we use $\langle c_1, s_1 \rangle \cdot \langle
c_1, s_2 \rangle$ to denote $\langle \langle c_1, s_1 \rangle, \langle c_2,
s_2 \rangle \rangle$ when there is no confusion.

We have the following lemma.

\begin{lemma}
  For any spinor $c \in \mathbb{C}^m$ and $m$-tuple $s$ of spinors
  \[ \langle c, \omega_{N_0} s \rangle = \langle \omega_{N_0} c, s \rangle .
  \]
\end{lemma}

\begin{proof}
  We proceed by writing in components,
  \[ \langle c, \omega_{N_0} s \rangle = \sum_{\alpha, \beta} \bar{c}_\alpha
     (\omega_{N_0})_{\alpha \beta} s_\beta = \sum_{\alpha, \beta}
     \overline{\overline{(\omega_{N_0})_{\alpha \beta}} c_\alpha} s_\beta . \]
  Since $\omega_{N_0}$ is Hermitian, so $\overline{(\omega_{N_0})_{\alpha \beta}} =
  (\omega_{N_0})_{\beta \alpha}$. So
  \[ \langle c, \omega_{N_0} s \rangle = \sum_{\alpha \beta}
     \overline{(\omega_{N_0})_{\beta \alpha} c_\alpha} s_\beta = \langle \omega_{N_0} c, s
     \rangle . \]
  
\end{proof}

Now we show that the components of $s$ that we obtained are linearly independent.

\begin{proposition} \label{full independence}
  The components of $s$ are linearly independent.
\end{proposition}

We show first that the components of $(1 + \omega_{N_0}) s$ and $(1 -
\omega_{N_0}) s$ are respectively linearly independent, and then we finish the
proof by showing that any pair of components with one from $(1 + \omega_{N_0})
s$ and the other from $(1 - \omega_{N_0}) s$ are orthogonal. To this end, we
introduce the following set
\begin{equation}
  L = \{c \in \mathbb{C}^m : \text{ } \langle c, (1 + \omega_{N_0}) s \rangle
  = \langle c, (1 - \omega_{N_0}) s \rangle = 0 \text{ everywhere on }
  \Omega\} \label{L set}
\end{equation}
and we have the following lemma.

\begin{lemma}\label{L vanishes}
  The set $L$ defined in \eqref{L set} is $\{0\}$.
\end{lemma}

\begin{remark}
Note that $\b{\epsilon} \bar{c}\left(N_0\right) \in \operatorname{End}\left(S_{\Omega_\delta}\right)$, which is Hermitian symmetric and $\left(\bar{\epsilon} \bar{c}\left(N_0\right)\right)^2=\mathrm{id}$. We can choose the basis  $\left\{{\bar{s}}_\alpha\right\}_{1 \leq \alpha \leq m}$  of spinors in $S_{\Omega_{\delta}}$ such that $\{\bar{s}_{\alpha}\}_{1 \le \alpha \le m}$ are the the eigenvectors of $\bar{\epsilon} \bar{c}\left(N_0\right)$, that is, they satisfy
$$
\bar{\epsilon}\b{c}\left(N_0\right){\bar{s}}_\alpha=\lambda_\alpha {\bar{s}}_\alpha, \quad 1 \leq \alpha \leq m,
$$
where
$$
\lambda_\alpha= \begin{cases}1 & 1 \leq \alpha \leq \frac{m}{2} \\ -1 & \frac{m}{2}+1 \leq \alpha \leq m .\end{cases}
$$
Then
$$
\left(\omega_{N_0} {s}\right)_\alpha=\omega_{N_0 \alpha \beta} {s}_\beta=\left\langle\bar{\epsilon} \bar{c}\left(N_0\right) {\bar{s}}_\beta, {\bar{s}}_\alpha\right\rangle {s}_\beta =\left\langle {\bar{s}}_\beta, \bar{\epsilon}\bar{ c}\left(N_0\right){\bar{s}}_\alpha\right\rangle {s}_\beta=\lambda_\alpha {s}_\alpha .
$$
Hence $L$ can be written as
$$
L=\left\{c \in \mathbb{C}^m: \sum_{\alpha=1}^{\frac{m}{2}} \bar{c}_\alpha {s}_\alpha=0=\sum_{\beta=\frac{m}{2}+1}^m \bar{c}_\beta {s}_\beta \text{ everywhere on } \Omega\right\}.
$$
Then $L=\{0\}$ if and only if each of the following two sets
\begin{equation}
S_1=\{{s}_1, \cdots, {s}_{\frac{m}{2}}\}, \quad S_2=\{{s}_{\frac{m}{2}+1}, \cdots, {s}_m\} \label{S1 and S2}
\end{equation}
is a linearly independent set at each point of $\Omega$.
  \end{remark}

\begin{proof}[Proof of Lemma \ref{L vanishes}]
  First, we note that if $\langle c, (1 \pm \omega_{N_0}) s \rangle$ vanishes
  at some point of $\Omega$, then it vanishes on all of $\Omega$. Indeed,
\begin{align}
& \nabla_{e_i} \langle c, (1 \pm \omega_{N_0}) s \rangle \\
= & \langle c, (1 \pm \omega_{N_0}) \nabla_{e_i} s \rangle \\
= & \tfrac{1}{2} \langle c, (1 \pm \omega_{N_0}) \epsilon c (e_i)
\omega_{N_0} s \rangle \\
= & \tfrac{1}{2} \epsilon c (e_i) \langle c, (1 \pm \omega_{N_0})
\omega_{N_0} s \rangle \\
= & \tfrac{1}{2} \epsilon c (e_i) \langle c, (\omega_{N_0} \pm 1) s
\rangle = 0.
\end{align}
  By an ODE argument, $\langle c, (1 \pm \omega_{N_0}) s
  \rangle$ vanishes everywhere on $\Omega$. Let $x_0 \in \partial \Omega$ be a
  point such that its inward-pointing unit normal is $N$. For any $c \in L$,
  we have
\begin{align}
& \langle \omega_N c, (1 \pm \omega_{N_0}) s \rangle \\
= & \langle c, \omega_N (1 \pm \omega_{N_0}) s \rangle \\
= & \langle c, (\omega_N \pm (2 \langle N, N_0 \rangle - \omega_N
\omega_{N_0})) s \rangle \\
= & \langle c, (1 \mp \omega_{N_0}) \omega_N s \rangle \pm \langle N, N_0
\rangle \langle c, s \rangle \\
= & \langle c, (1 \mp \omega_{N_0}) \omega_N s \rangle
\end{align}
  where we have used that $\langle c, s \rangle = 0$. Using the boundary
  condition $\epsilon c (\nu) \omega_N s = - s$, we have $\omega_N s = -
  \epsilon c (\nu) s$ and
  \[ \langle \omega_N c, (1 \pm \omega_{N_0}) s \rangle = \langle c, (1 \mp
     \omega_{N_0}) \omega_N s \rangle = - \epsilon c (\nu) \langle c, (1 \pm
     \omega_{N_0}) s\rangle = 0 \]
  at $x_0$. Arguing similarly as before, $\langle \omega_N c, (1 \pm
  \omega_{N_0}) s \rangle = 0$ on all of $\Omega$, so $\omega_N c\in L$. Since the linear span of
  the unit normals of $\partial \Omega$ is $\mathbb{R}^n$ and
  $\ensuremath{\operatorname{End}} (\mathbb{C}^m)$ is generated by all such
  $\omega_N$, $L$ is invariant under all $\ensuremath{\operatorname{End}}
  (\mathbb{C}^m)$. It implies that $L$ is either $\{0\}$ or $\mathbb{C}^m$.
  The latter is impossible since $s$ has at least one non-zero component.
\end{proof}
Next, we show that the two sets $S_1$ and $S_2$ are orthogonal.
\begin{lemma} \label{cross orthogonal}
For any ${s}_\alpha\in S_1$ and ${s}_\beta\in S_2$, $\langle s_{\alpha}, s_{\beta} \rangle =0$.
  \end{lemma}
  \begin{proof}
For any ${s}_\alpha\in S_1$ and ${s}_\beta\in S_2$, 
\begin{align}
\begin{split}
{\xi}\left\langle {s}_\alpha,{s}_\beta\right\rangle&=\left\langle \n_{{\xi}}{s}_\alpha,{s}_\beta\right\rangle+\left\langle {s}_\alpha,\n_{{\xi}}{s}_\beta\right\rangle\\
  &=\left\langle \frac{1}{2}{\epsilon}{c}({\xi}){s}_\alpha,{s}_\beta\right\rangle+\left\langle {s}_\alpha,-\frac{1}{2}{\epsilon}{c}({\xi}){s}_\beta\right\rangle\\
  &=0,
 \end{split}
\end{align}
which follows that $\left\langle {s}_\alpha,{s}_\beta\right\rangle$ is a constant in $\Omega$. 
Let $x_0\in \partial\Omega$ be the point such that its inward-pointing unit normal vector is $N=N_0$. Then 
\begin{align}
\begin{split}
  \left\langle {s}_\alpha,{s}_\beta\right\rangle&= \left\langle {s}_\alpha,-{\epsilon}{c}(\nu)(\omega_{N_0}{s})_\beta\right\rangle\\
  &= \left\langle {s}_\alpha,{\epsilon}{c}(\nu){s}_\beta\right\rangle\\
  &=\left\langle {\epsilon}{c}(\nu){s}_\alpha,{s}_\beta\right\rangle\\
  &=\left\langle {\epsilon}{c}(\nu)(\omega_{N_0}{s})_\alpha,{s}_\beta\right\rangle\\
  &=\left\langle -{s}_\alpha,{s}_\beta\right\rangle,
 \end{split}
\end{align}
which implies that $ \left\langle {s}_\alpha,{s}_\beta\right\rangle(x_0)=0$, and so 
\begin{equation}
  \left\langle {s}_\alpha,{s}_\beta\right\rangle\equiv0
\end{equation}
on all of $\Omega$.
\end{proof}

\begin{proof}[Proof of Proposition \ref{full independence}]
  It follows from Lemma \ref{L vanishes} that $S_1$ and $S_2$ defined in \ref{S1 and S2} are respectively linearly independent. It follows from Lemma \ref{cross orthogonal} that $S_1 \cup S_2$ are linearly independent. So the components of $s$ are linearly independent.
  \end{proof}

  Now we are ready to prove Theorem \ref{hyperbolic Llarull}.

  \begin{proof}[Proof of Theorem \ref{hyperbolic Llarull}]
For any $x\in \Omega$, let $\{e_i\}_{1\leq i\leq n}$ be an orthonormal basis of $(T_x\Omega,g)$, by \eqref{eqn-10}, we have
$$
\begin{aligned}
0= & \nabla_{e_k}\left(\nabla_{e_l} {s}-\frac{1}{2} \omega_{N_0} {\epsilon}{c}\left( e_l\right) {s}\right)-\left(\nabla_{\nabla_{e_k} e_l} s-\frac{1}{2} \omega_{N_0} {\epsilon}{c}\left( \nabla_{e_k} e_l\right) {s}\right) \\
& +\frac{1}{2} {\epsilon}{c}\left( e_l\right) \omega_{N_0}\left(\nabla_{e_k} {s}-\frac{1}{2} \omega_{N_0} {\epsilon}{c}\left(e_q\right) s\right) \\
= & \nabla_{e_k} \nabla_{e_l}{ s}-\nabla_{\nabla_{e_k} e_l} {s}+\frac{1}{4} {c}\left(e_l\right) {c}\left(e_q\right) {s}.
\end{aligned}
$$
This implies
$$
\begin{aligned}
0= & \nabla_{e_k} \nabla_{e_l} {s}-\nabla_{e_l} \nabla_{e_k} {s}-\nabla_{\left[e_k, e_l\right]} {s} \\
& -\frac{1}{4} \sum_{i, j=1}^n\left(\delta_{i k} \delta_{j l}-\delta_{i l} \delta_{j k}\right) {c}\left(e_i\right) {c}\left(e_j\right){ s}.
\end{aligned}
$$
Hence
$$
\begin{aligned}
0 & =-\frac{1}{4} \sum_{i, j=1}^n\left(-\left\langle R\left(e_k, e_l\right) e_i, e_j\right\rangle+\left(\delta_{i k} \delta_{j l}-\delta_{i l} \delta_{j k}\right)\right) {c}\left(e_i\right) {c}\left(e_j\right){ s} \\
& =-\frac{1}{4} \sum_{i, j=1}^n\left(R\left(e_i, e_j, e_k, e_l\right)+\left(\delta_{i k} \delta_{j l}-\delta_{i l} \delta_{j k}\right)\right) {c}\left(e_i\right){ c}\left(e_j\right) {s},
\end{aligned}
$$
which implies that
$$
\sum_{i, j=1}^n\left(R\left(e_i, e_j, e_k, e_l\right)+\left(\delta_{i k} \delta_{j l}-\delta_{i l} \delta_{j k}\right)\right) {c}\left(e_i\right) {c}\left(e_j\right) {s}_\mu=0
$$
for any $1\leq \mu\leq m$.
Note that $\{{s}_\alpha\}_{1\leq \alpha\leq m}$ is a basis of $S_{\Omega_g}$ by 
Proposition \ref{full independence}, then 
\begin{equation}
  {c}\left(\sum_{i, j=1}^n\left(R\left(e_i, e_j, e_k, e_l\right)+\left(\delta_{i k} \delta_{j l}-\delta_{i l} \delta_{j k}\right)\right)e_i\cdot e_j \right)=0\in \mathrm{End}(S_{\Omega_g}|_{x}).
\end{equation}
For $n$ is even, the representation ${c}:\mathrm{Cl}(T_x\Omega,\mb{C})\to \mathrm{End}(S_{\Omega_g}|_{x})$ is an isomorphism, see for example \cite[Theorem 1.28]{BHMMM}. It follows that
$$R\left(e_i, e_j, e_k, e_l\right)+\left(\delta_{i k} \delta_{j l}-\delta_{i l} \delta_{j k}\right)=0,$$
that is, $(\Omega,g)$ is hyperbolic. The proof of Theorem \ref{hyperbolic Llarull} is complete.
\end{proof}

\section{Scalar curvature rigidity for polytopes in the hyperbolic space}\label{sec:dihedral}

In this section, we briefly introduce
the smoothing procedures of Brendle \cite{brendle-scalar-2024}, which we use 
to find a non-zero solution to \eqref{eqn-10} subject to the boundary condition \eqref{bc brendle}.
This leads to a proof of Theorem \ref{intersect rigidity} assuming matching angles.

\subsection{Smoothing}

Let $\Omega$ be a compact convex polytope in $\mathbb{R}^n$ with non-empty
interior. We write $\Omega = \cap_{\ell = 1} \{u_{\ell} \leq 0\}$, where
$u_{\ell}$, $\ell \in \Lambda$ is a finite collection of non-constant linear
functions defined on $\mathbb{R}^n$. Without loss of generality, we can assume
that for each $\ell_0 \in \Lambda$, the set $\{u_{\ell_0} > 0\} \cap (\cap_{\ell \in
\Lambda\backslash\{\ell_0 \}} \{u_{\ell} \leq 0\})$ is non-empty. For
sufficiently large $\lambda$, Brendle's smoothing \cite{brendle-scalar-2024} of $\Omega$ is given by
\begin{equation} \Omega_{\lambda} = \left\{ \sum_{\ell} e^{\lambda u_{\ell}} \leq 1
  \right\}\label{brendle smoothing explicit}
\end{equation}
for sufficiently large $\lambda>0$.
 Let $N_{\lambda}:\partial\Omega_{\lambda}\to \mathbb{S}^{n-1}$ be given by
 \begin{equation}
 N_{\lambda} = \sum_{\ell \in \Lambda} e^{\lambda u_{\ell}} | \nabla
  u_{\ell} | \nu_\ell \left| \sum_{\ell \in \Lambda} e^{\lambda u_{\ell}} |
  \nabla u_{\ell} | \nu_\ell \right|^{- 1} \label{homotopic N}.
\end{equation}
It is clear that $N_{\lambda}$ is homotopic to the Euclidean Gauss map of 
$\partial\Omega_{\lambda}$ by simply deforming the metric ${g}$ to the flat metric.

The map $N_{\lambda} : \partial \Omega_{\lambda} \to \mathbb{S}^{n - 1}$ is homotopic
to the Euclidean Gauss map of $\partial \Omega_{\lambda}$, we define the 
bundle map $ \chi_{\lambda}$ as
\begin{equation}
  \chi_{\lambda} \sigma : = ({\epsilon} \otimes \bar{\epsilon}) \cdot ({c} ({e}_n) \otimes \bar{c} (N_{\lambda})) \sigma . \label{chi lambda}
\end{equation}
It is easy to check that $\chi_{\lambda}$ is self-adjoint and $\chi_{\lambda}^2$ is the identity
map.

\subsection{Integral formula with an N homotopic to the Euclidean
Gauss map}
In the hyperbolic case \(\psi(x^1)=1/x^1\), so \(\psi'/\psi^2=-1\) and
\begin{equation}
  \Psi = - \tfrac{n}{2} ({\epsilon} \otimes \bar{\epsilon}) \cdot (1
  \otimes \bar{c} (\partial_{x^1})),\quad \partial_{x^1}:=\tfrac{\partial}{\partial x^1}. \label{psi in dihedral}
\end{equation}
Using $\psi (x^1)=1/x^1$ in the proof of Proposition \ref{Dirac Witten}, we obtain
\begin{align}
  & \int_{\Omega_{\lambda}} |\hat{D}\sigma |^2 \\
  = &  \int_{\Omega_{\lambda}} |\hat{\nabla} \sigma |^2 +
  \tfrac{1}{4} \int_{\Omega_{\lambda}} (R_{{g}} + n (n - 1)) |
  \sigma |^2 \label{hyperbolic prelim} \\
  & \quad + \int_{\partial \Omega_{\lambda}} \left[  \langle
  ({c} ({e}_n) D + \nabla_{{e}_n}) \sigma, \sigma \rangle +
 \tfrac{n-1}{n} \langle {c} ({e}_n) \Psi \sigma, \sigma \rangle \right] . 
\end{align}
Now we assume that $\sigma$ satisfies the following boundary condition
\begin{equation}
  \chi_{\lambda} \sigma = - \sigma \label{homotopic boundary condition}
\end{equation}
where $\chi_{\lambda}$ is defined in \eqref{chi lambda}.

Let $Y$ be a Euclidean vector field at a point $x_0 \in \partial
\Omega_{\lambda}$ such that $Y$ is orthogonal to $N_{\lambda}$ at $x_0$. Since
$({\epsilon} \otimes \bar{\epsilon}) \cdot ({c} ({e}_n) \otimes \bar{c}
(N_{\lambda}))$ and $({\epsilon} \otimes \bar{\epsilon}) \cdot ({c} ({e}_n)
\otimes \bar{c} (Y))$ are both self-adjoint and anti-commute with each other,
\[ \langle ({\epsilon} \otimes \bar{\epsilon}) \cdot ({c} ({e}_n)
   \otimes \bar{c} (Y)) \sigma, \sigma \rangle = 0 \]
at $x_0$. Using the above, \eqref{psi in dihedral} and \eqref{homotopic
boundary condition}, we have
\begin{align}
  & \langle {c} ({e}_n) \Psi \sigma, \sigma \rangle \\
  = & - \tfrac{n}{2} \langle {c} ({e}_n) ({\epsilon} \otimes
  \bar{\epsilon}) \cdot (1 \otimes \bar{c} (\partial_{x^1})), \sigma \rangle \\
  = & - \tfrac{n}{2} \langle {c} ({e}_n) ({\epsilon} \otimes
  \bar{\epsilon}) \cdot (1 \otimes \bar{c} (\langle \partial_{x^1}, N_\lambda \rangle N_\lambda)), \sigma
  \rangle \\
  = & \tfrac{n}{2} \langle \partial_{x^1}, N_\lambda \rangle \langle {c} ({e}_n)
  ({\epsilon} \otimes \bar{\epsilon}) \cdot (1 \otimes \bar{c} (N_\lambda)),
  ({\epsilon} \otimes \bar{\epsilon}) \cdot ({c} ({e}_n) \otimes \bar{c}
  (N_\lambda)) \sigma \rangle \\
  = & - \tfrac{n}{2} \langle \partial_{x^1}, N_\lambda \rangle . \label{hyp psi term}
\end{align}
Now we define some auxiliary connections
\begin{equation}
  \bar{\nabla}^{\partial}_{e_j} = \bar{\nabla}_{e_j} + \tfrac{1}{2} \bar{c}
  (\bar{\nabla}_{e_j} N_\lambda) \bar{c} (N_\lambda), \text{ } \nabla^{g, \partial}_{e_j} =
  \nabla_{e_j} + \tfrac{1}{2} c (\nabla_{e_j} e_n) c (e_n), \label{homotopic connection}
\end{equation}
where $\bar{\nabla}$ is the Levi-Civita connection on the spinors or vector fields with respect
to the Euclidean metric and $\nabla$ denotes the connection on the spinors in
$S_{\Omega_g}$. 
Similar to the boundary Dirac operator \eqref{boundary dirac operator}, we define
\begin{equation}
  \nabla^{\partial} = \nabla^{g, \partial} \otimes 1 + 1 \otimes
  \bar{\nabla}^{\partial}, \text{ } \tilde{D}^{\partial} = \sum_{j = 1}^{n -
  1} c (e_n) c (e_j) \tilde{\nabla}^{\partial}_{e_j} .
\end{equation}
It is tedious but
direct to check that $\tilde{D}^{\partial}$ and $\chi$ anti-commute. Hence
$\langle \tilde{D}^{\partial} \sigma, \sigma \rangle = 0$ since $\sigma$
satisfies \eqref{homotopic boundary condition}. We see
\begin{align}
  & \langle c (e_n) D \sigma + \nabla_{e_n} \sigma, \sigma \rangle
  \nonumber\\
  = & \sum_{i = 1}^{n - 1} \langle c (e_n) c (e_i) \nabla_{e_i} \sigma, \sigma
  \rangle \nonumber\\
  = & \langle \tilde{D}^{\partial} \sigma, \sigma \rangle - \tfrac{1}{2}
  \sum_{i = 1}^{n - 1} \langle c (e_n) c (e_i) \bar{c} (\mathrm{d} N_\lambda (e_i))
  \bar{c} (N_\lambda) \sigma, \sigma \rangle \nonumber\\
  & \quad - \tfrac{1}{2} \sum_{i = 1}^{n - 1} \langle c (e_n) c (e_i) c
  (\nabla_{e_i} e_n) c (e_n) \sigma, \sigma \rangle \nonumber\\
  = & \tfrac{1}{2} H | \sigma |^2 - \tfrac{1}{2} \sum_{i, j = 1}^{n - 1}
  \langle \mathrm{d} N_\lambda (e_i), \tilde{e}_j \rangle \langle c (e_n) c (e_i)
  \bar{c} (\tilde{e}_j) \bar{c} (N_\lambda) \sigma, \sigma \rangle, \nonumber
\end{align}
where $\{\tilde{e}_{j}\}_{j=1,2,\ldots,n-1}$ is an orthonormal basis of the subspace orthogonal to
$N_\lambda$. Let $q_i \geqslant 0$ be the \textit{singular values} of the differential
$\mathrm{d} N_\lambda : T_{x_0} \partial \Omega_{\lambda} \to T_{N_\lambda (x_0)}
\mathbb{S}^{n - 1}$ of $N$, that is, we can choose an orthonormal basis of
$T_{x_0} \partial \Omega_{\lambda}$ and $T_{N_\lambda (x_0)} \mathbb{S}^{n - 1}$ such that
\[  \mathrm{d}N_\lambda(e_i) = \bar{\nabla}_{e_{i}} N_\lambda = q_{i} \tilde{e}_{i} .\]
Also, note that we have used $\bar{\nabla} N_\lambda = \mathrm{d} N_\lambda$.
The \textit{trace norm} $\|\mathrm{d}N_\lambda\|_{\mathrm{tr}}$ of $\mathrm{d}N_\lambda$ is given by
\[\|\mathrm{d}N_\lambda\|_{\mathrm{tr}} = \sum_{i=1}^{n-1}q_i.\]
Hence,
\begin{align}
  & \langle (c (e_n) D + \nabla_{e_n}) \sigma, \sigma \rangle \\
  = & \tfrac{1}{2} H| \sigma |^2 - \sum_{j = 1}^{n - 1} \tfrac{1}{2} q_{j} 
  \langle c (e_n) c (e_j) c (\tilde{e}_j) c (N_\lambda) \sigma, \sigma \rangle, \\
  \geq & \tfrac{1}{2} H| \sigma |^2 - \tfrac{1}{2} \sum_{j = 1}^{n - 1} q_j |
  \sigma |^2 \\
  = & \tfrac{1}{2} (H -\| \mathrm{d} N_\lambda\|_{\mathrm{tr}}) | \sigma |^2 . 
\end{align}
Putting \eqref{hyp psi term} and the above into \eqref{hyperbolic prelim}, we
obtain
\begin{align} \label{sl for polytope epsilon}
   \int_{\Omega_{\lambda}} | \hat{D} \sigma |^2 \geq &  \int_{\Omega_{\lambda}} | \hat{\nabla} \sigma |^2
  + \tfrac{1}{4 } \int_{\Omega_{\lambda}} (R_{{g}} + n (n - 1)) |
  \sigma |^2 \\
  & \quad + \tfrac{1}{2} \int_{\partial \Omega_{\lambda}} (H - (n -
  1) \langle \partial_{x^1}, N_\lambda \rangle -\| \mathrm{d}
  N_\lambda\|_{\ensuremath{\operatorname{tr}}}) | \sigma |^2 . 
\end{align}
It is worth to remark that $N_{\lambda}$ could be replaced with any map from $N':\partial \Omega_{\lambda}\to \mathbb{S}^{n-1}$, in particular, the Gauss map of $\partial \Omega_{\lambda}$ itself.
By index theory (see \eqref{index theory}) and using the deformation invariance of the Fredholm index, we can find a non-zero section $\sigma^{(\lambda)}$
defined on $\Omega_{\lambda}$ which satisfies
\[ \hat{D} \sigma^{(\lambda)} = 0 \text{ in } \Omega_{\lambda}, \text{ }
   ({\epsilon} \otimes \bar{\epsilon}) ({c}({e}_n) \otimes \bar{c}(N_{\lambda}))
   \sigma^{(\lambda)} = - \sigma^{(\lambda)} \text{ along } \partial \Omega_{\lambda} .
\]
Using the same argument in Brendle \cite{brendle-scalar-2024} (replacing $H$ by $H - (n-1)\mathrm{d}x^1(N)$), we can show there exists a subsequence $\{\lambda_l\}_{l\in \mathbb{N}}$ such that $\sigma^{(\lambda_{l})}$ converge in
$C^{\infty}_{\ensuremath{\operatorname{loc}}}$  away from the edges of $\Omega$ to
a non-zero section $\sigma$, $N_{\lambda_{l}}$ converge to $N$ and 
\begin{equation}
 \hat{D} \sigma =0 \text{ in } \Omega \label{imaginary dirac}
\end{equation}
subject to the boundary condition
\begin{equation}
  ({\epsilon} \otimes \bar{\epsilon}) ({c} ({\nu}) \otimes \bar{c} (N)) \sigma = - \sigma \text{ along }\partial{\Omega} . \label{boundary condition}
\end{equation}
The curcial inequality \eqref{sl for polytope epsilon} is also preserved in the limit (the right hand is uniformly bounded below), so using the scalar curvature bound $R_g \geqslant-n(n-1)$ and the mean curvature bound $H_g \geq (n-1) \langle N_0, N \rangle$ along a face $F$ in the polytope $\Omega$ (assumptions of Theorem \ref{intersect rigidity}), we conclude that the limit $\sigma$ must satisfy
\begin{equation}
 \hat{\nabla}_{\xi} \sigma =  \nabla_{{\xi}} \sigma - \tfrac{1}{2} ({\epsilon} \otimes
  \bar{\epsilon}) ({c} ({\xi}) \otimes \bar{c} (N_{0})) \sigma = 0 \text{ in }\Omega,
  \label{killing}
\end{equation}
which is stronger than \eqref{imaginary dirac}. 
Due to \eqref{killing}, $\sigma$ can be extended continuously to $\Omega$.

We fix a basis $\{\bar{s}_{\alpha}\}$ of $S_{\Omega_{\delta}}$, then $\sigma = \sum_{\alpha} s_{\alpha} \otimes \bar{s}_{\alpha}$ for some $\{s_{\alpha}\}_{1 \leqslant \alpha \leqslant m}$, and the $m$-tuple of spinors $s = (s_1,\ldots, s_{\alpha})$ satisfies
\begin{equation}\label{eqn-10 s}
  \nabla_{{\xi}}{s}-\frac{1}{2}\epsilon{c}({\xi})\omega_{N_0}{s}=0, 
\end{equation}
subject to the boundary condition $({\epsilon}\otimes \bar{\epsilon})({c}(\nu)\otimes \bar{c}(N))\sigma=-\sigma$, which is equivalent to 
\begin{equation}
  {\epsilon}{c}(\nu)\omega_N({s})=-{s}, \label{bc brendle s}
\end{equation}
As before, the $m$-tuple of spinors $s$ are easier to work with,
and the equations \eqref{eqn-10 s} and \eqref{bc brendle s} are the starting point of our analysis.

\subsection{Linear independence of components of s}

\

To show that $(\Omega, g)$ is hyperbolic as in the Proof of Theorem
\ref{hyperbolic Llarull}, we need that the components of $s$ are linearly
independent which in turn relies on Lemma \ref{cross orthogonal}. We have the
following proposition which is a restatement of Lemma \ref{cross orthogonal}
since $\omega_{N_0}$ is self-adjoint with respect to the formal inner product
given in \eqref{c s}. However, here $\Omega$ is a polytope and it is not
always true we can find a face $F$ with its normal $N = \pm N_0$. To achieve
Proposition \ref{lm:goal}, we use a simple geometric argument.

From now on, to avoid double levels of angular brackets, we use $\langle c_1,
s_1 \rangle \cdot \langle c_1, s_2 \rangle$ to denote $\langle \langle c_1,
s_1 \rangle, \langle c_2, s_2 \rangle \rangle$ when there is no confusion.

\begin{lemma}
  \label{4.8}Let $F$ be a face of $\Omega$ and $N$ be its Euclidean normal.
  Then for any $c_i \in \mathbb{C}^m$, $i = 1, 2$ that
  \begin{equation}
    \langle c_1, (1 + \omega_N) s \rangle \cdot \langle c_2, (1 - \omega_N) s
    \rangle = 0 \label{cor from bv}
  \end{equation}
  along $F$.
\end{lemma}

\begin{proof}
  The proof is a direct calculation by using $\omega_N s = - \epsilon c
  (\nu) s$, $(\epsilon c (\nu))^2 = 1$ and that $\epsilon c (\nu)$ is
  self-adjoint.
\end{proof}

\begin{proposition}
  \label{lm:goal}Let $c_i \in \mathbb{C}^m$, $i = 1, 2$ be the vectors 
  satisfying $\omega_{N_0} c_i = (- 1)^{i - 1} c_i$, then
  \begin{equation}
    \langle c_1, s \rangle \cdot \langle c_2, s \rangle = 0. \label{goal}
  \end{equation}
  Moreover, each component of $s$ is of type I, see Definition \ref{defn-typeIII} for the definition of type I spinors.
\end{proposition}

\begin{proof}
Let $\Omega$ be a convex polytope in $(\mb{R}^n_+,\delta)$, let $p_0\in \partial\Omega$ be the (highest) point in $\Omega$, that is,
\begin{equation}
  x^1(p_0)=\max_{p\in \Omega}x^1(p).
\end{equation}
We assume that $p_0$ is in the codimension $k$ stratum of $\Omega$ for some $k$, then it admits a small neighborhood $U$ of the form:
\begin{equation}
  \mb{R}^{n-k}\times P
\end{equation}
such that $P$ is a polyhedral corner in $\mathbb{R}^k$ enclosed by hyperplanes passing through the origin of $\mathbb{R}^k$ and $p_0$ is the origin of $\mathbb{R}^n$. 

Since $x^1(p_0)\geq x^1(p)$ for any $p\in \Omega$, we conclude that the $(n-k)$-dimensional plane $\mb{R}^{n-k}$ in $\mb{R}^{n-k}\times P$ is orthogonal to $N_0=\frac{\p}{\p x^1}$. Note that $P$ is a polyhedral corner in $\mathbb{R}^k$, the intersection of all tangent spaces of $P$ at $p_0$ is the origin $p_0$, so
the space spanned by its normal vectors of these tangent spaces is $\mb{R}^k$. Otherwise, if a non-zero vector perpendicular to all normal vectors exists, it must fall in all tangent spaces of $P$, which is a contradiction since the intersection is a point. 
Hence, $N_0$ can be spanned by these normal vectors, that is, 
\begin{equation}
  N_0=\sum_{i=1}^{i_0}a_iN_i,
\end{equation}
where $N_i,1\leq i\leq i_0$, are the outward-pointing unit normal vector of hyperplanes in $P$. 

At this point $p_0\in \Sigma$, the boundary condition gives 
\begin{equation}
  \epsilon c(\nu_i)\omega_{N_i}s=-s,\quad 1\leq i\leq i_0,
\end{equation}
which is equivalent to 
\begin{equation}
 \omega_{N_i}s= -\epsilon c(\nu_i) c(\nu_i)s,
\end{equation}
and so 
\begin{equation}
   \omega_{N_0}s= \sum_{i=1}^{i_0}a_i\omega_{N_i}s=-\sum_{i=1}^{i_0}a_i \epsilon c(\nu_i)s=-\epsilon c(\nu_0)s,
\end{equation}
where
 $\nu_0:=\sum_{i=1}^{i_0}a_i\nu_i$.
In particular, we have
\begin{equation}\label{eqn1}
   -\epsilon c(\nu_0)s_\alpha=s_\alpha,\quad s_\alpha\in S_1=\{s_1,\cdots,s_{\frac{m}{2}}\}.
\end{equation}
and 
\begin{equation}\label{eqn2}
   \epsilon c(\nu_0)s_\beta=s_\beta,\quad s_\beta\in S_2=\{s_{\frac{m}{2}+1},\cdots,s_m\}.
\end{equation}
By acting $-\epsilon c(\nu_0)$ to both sides,
 \begin{equation}
 |\nu_0|^2 s_\alpha= -\epsilon c(\nu_0)s_\alpha=s_\alpha,
\end{equation}
so $\nu_0$ is a unit vector since $s_\alpha(p_0)\neq 0$. By \eqref{eqn1} and \eqref{eqn2}, $s_\alpha$ and $s_\beta$ are of type I by Remark \ref{rem-typeI}.

For any $s_\alpha\in S_1$ and $s_\beta\in S_2$, at this point $p_0\in \Sigma$,
\begin{align}
\left\langle s_\alpha,s_\beta\right\rangle&=\left\langle -\epsilon c(\nu_0)s_\alpha,s_\beta\right\rangle\\
&=-\left\langle s_\alpha,\epsilon c(\nu_0)s_\beta\right\rangle\\
&=-\left\langle s_\alpha,s_\beta\right\rangle
\end{align}
which follows that
$\left\langle s_\alpha,s_\beta\right\rangle(p_0)=0.$
Since $\left\langle s_\alpha,s_\beta\right\rangle$ is constant, 
$\left\langle s_\alpha,s_\beta\right\rangle\equiv 0.$
The proof is complete.
  \end{proof}

\subsection{Orthogonality of spinor components}

Now we establish the following orthogonality of the spinor components of $s$.

\begin{proposition}
  \label{scalar identity}Let $G = (G_{\alpha \beta})_{1 \leq \alpha,
  \beta \leq m}$ be the matrix given by $G_{\alpha \beta} = \langle
  s_{\alpha}, s_{\beta} \rangle$. Then $G$ is a non-zero multiple of the
  identity matrix everywhere. Note that $|s_{\alpha} |$ is not a constant.
\end{proposition}

Let $\Lambda_{\pm} = \{c \in \mathbb{C}^m : \text{ } \omega_{N_0} c = \pm
c\}$. We have the following lemma.

\begin{lemma}
  Let $F$ be a face of $\Omega$ and $N$ be its Euclidean unit normal and
  $X = N - \langle N_0, N \rangle N_0$, then
   \begin{equation}
    \langle c_1, s \rangle \cdot \langle c_2, \omega_X s \rangle = \langle
    c_1, \omega_X s \rangle \cdot \langle c_2, s \rangle \label{relation 2}
  \end{equation}
 $c_1 \in \Lambda_+$ and $c_2 \in \Lambda_-$.
\end{lemma}

\begin{proof}
We start by setting
  $a = \langle N, N_0 \rangle$ and $b = \langle N, X \rangle$, then $N = a N_0
  + b X$ and $a^2 + b^2 = 1$. From \eqref{cor from bv}, we see
  \[ 0 = \langle c_1, (1 + a) s + b \omega_X s \rangle \cdot \langle c_2, (1 +
     a) s - b \omega_X s \rangle . \]
  Expanding leads to
\begin{align}
0 = & (1 + a)^2 \langle c_1, s \rangle \cdot \langle c_2, s \rangle - b^2
\langle c_1, \omega_X s \rangle \cdot \langle c_2, \omega_X s \rangle
\\
& \qquad - b (1 + a) [\langle c_1, s \rangle \cdot \langle c_2, \omega_X
s \rangle - \langle c_1, \omega_X s \rangle \cdot \langle c_2, s \rangle].
\end{align}
We know from Lemma \ref{lm:goal} that $\langle c_1,s \rangle \cdot \langle c_2,s \rangle =0$.
 Since $\omega_{N_0} c_1 = c_1$, $X$ and $N_0$ are
  orthogonal, we have that $\omega_{N_0} \omega_X c_1 = - \omega_X
  \omega_{N_0} c_1 = - \omega_X c_1$. So $\omega_X c_1 \in \Lambda_-$, and
  similarly, $\omega_X c_2 \in \Lambda_+$. Hence $\langle c_1, \omega_{X}s \rangle \cdot \langle c_2, \omega_{X}s \rangle =0$ as well, and the lemma is proved.
\end{proof}

We have the following lemma regarding the Hessian and boundary derivatives
of 
$ \langle c_1, s \rangle \cdot \langle c_2, s \rangle$.

\begin{lemma} \label{hessian and boundary derivative}
Let $c_1, c_2 \in \mathbb{C}^m$, then
\[ \nabla^{2} (\langle c_1, s \rangle \cdot \langle c_2, s \rangle) =
   \langle c_1, s \rangle \cdot \langle c_2, s \rangle g \text{ in }
   \Omega, \]
and
\[ \nabla_{\nu} (\langle c_1, s \rangle \cdot \langle c_2, s \rangle) =
   \langle N_0, N \rangle \langle c_1, s \rangle \cdot \langle c_2, s \rangle
   \text{ along } \partial \Omega . \]
  \end{lemma}

  \begin{proof}

First, we choose a geodesic normal frame at point $x \in
\Omega$ such that $\nabla_i e_j = 0$ and $\langle e_i , e_j \rangle = \delta_{ij}$
at the point $x$. It suffices to prove the lemma at $x$.
By differentiating the Killing
spinor equation twice we get that $\nabla_j \nabla_i s = - \tfrac{1}{4} e_i
\cdot e_j \cdot s$. We have
\begin{align}
& \nabla_j \nabla_i (\langle c_1, s \rangle \cdot \langle c_2, s \rangle)
\\
= & \langle c_1, \nabla_j \nabla_i s \rangle \cdot \langle c_2, s \rangle +
\langle c_1, s \rangle \cdot \langle c_2, \nabla_j \nabla_i s \rangle
\\
& \quad + \langle c_1, \nabla_j s \rangle \cdot \langle c_2, \nabla_i s
\rangle + \langle c_1, \nabla_i s \rangle \cdot \langle c_2, \nabla_j s
\rangle \\
= & - \langle c_1, \tfrac{1}{4} e_i \cdot e_j \cdot s \rangle \cdot \langle
c_2, s \rangle - \langle c_1, s \rangle \cdot \langle c_2, \tfrac{1}{4} e_i
\cdot e_j \cdot s \rangle \\
& \quad + \tfrac{1}{4} \langle c_1, \omega_{N_0} \epsilon e_j \cdot s
\rangle \cdot \langle c_2, \omega_{N_0} \epsilon e_i \cdot s \rangle
\\
& \quad + \tfrac{1}{4} \langle c_1, \omega_{N_0} \epsilon e_i \cdot s
\rangle \cdot \langle c_2, \omega_{N_0} \epsilon e_j \cdot s \rangle
\\
= & - \langle c_1, \tfrac{1}{4} (e_j \cdot e_i + e_i \cdot e_j) \cdot s
\rangle \cdot \langle c_2, s \rangle \\
& \quad - \tfrac{1}{4} \langle c_1, \omega_{N_0} (e_j \cdot e_i + e_i \cdot
e_j) \cdot s \rangle \cdot \langle c_2, \omega_{N_0} s \rangle \\
= & \tfrac{1}{2} (\langle c_1, s \rangle \cdot \langle c_2, s \rangle +
\langle c_1, \omega_{N_0} s \rangle \cdot \langle c_2, \omega_{N_0} s
\rangle) \delta_{i j} .
\end{align}
By considering the decomposition that $c_i = \tfrac{1}{2} (1 + \omega_{N_0})
c_i + \tfrac{1}{2} (1 - \omega_{N_0}) c_i$ and 
Lemma \ref{lm:goal}, we see 
\[ \langle c_1, \omega_{N_0} s \rangle \cdot \langle c_2, \omega_{N_0} s
   \rangle = \langle c_1, s \rangle \cdot \langle c_2, s \rangle . \]
Hence
\[ \nabla_j \nabla_i (\langle c_1, s \rangle \cdot \langle c_2, s \rangle) =
   \langle c_1, s \rangle \cdot \langle c_2, s \rangle \delta_{i j} .
   \label{staticity} \]

Let $F \subset \partial \Omega$ be a face of the polyhedron $\Omega$. We
compute the normal derivative $\nabla_{\nu} (\langle c_1, s \rangle \cdot
\langle c_2, s \rangle)$. We see that using the imaginary Killing equation and
the boundary condition,
\begin{align}
& \nabla_{\nu} (\langle c_1, s \rangle \cdot \langle c_2, s \rangle)
\\
= & - \tfrac{1}{2} \langle c_1, \omega_{N_0} \epsilon \nu \cdot s \rangle
\cdot \langle c_2, s \rangle - \tfrac{1}{2} \langle c_1, s \rangle \cdot
\langle c_2, \omega_{N_0} \epsilon \nu \cdot s \rangle \\
= & - \tfrac{1}{2} \langle c_1, \omega_{N_0} \omega_N s \rangle \cdot
\langle c_2, s \rangle - \tfrac{1}{2} \langle c_1, s \rangle \cdot \langle
c_2, \omega_{N_0} \omega_N s \rangle .
\end{align}
We use the decomposition $N = a N_0 + b X := \langle N_0, N \rangle N + b X$,
\begin{align}
& \nabla_{\nu} (\langle c_1, s \rangle \cdot \langle c_2, s \rangle)
\\
= & - \tfrac{1}{2} \langle c_1, a s + b \omega_{N_0} \omega_X s \rangle
\cdot \langle c_2, s \rangle - \tfrac{1}{2} \langle c_1, s \rangle \cdot
\langle c_2, a s + b \omega_{N_0} \omega_X s \rangle \\
= & - a \langle c_1, s \rangle \cdot \langle c_2, s \rangle . \\
& \quad - \tfrac{1}{2} b \langle \omega_X \omega_{N_0} c_1, s \rangle \cdot
\langle c_2, s \rangle - \tfrac{1}{2} b \langle c_1, s \rangle \cdot \langle
\omega_X \omega_{N_0} c_2, s \rangle \\
= & - a \langle c_1, s \rangle \cdot \langle c_2, s \rangle . \\
& \quad - \tfrac{1}{2} b (\langle \omega_X c_1^+, s \rangle \cdot \langle
c_2^-, s \rangle - \langle \omega_X c_1^-, s \rangle \cdot \langle c_2^+ s
\rangle) \\
& \quad - \tfrac{1}{2} b (- \langle c_1^+, s \rangle \cdot \langle \omega_X
c_2^- \rangle + \langle c_1^-, s \rangle \cdot \langle \omega_X c_2^+ s
\rangle)
\end{align}
where in the last line, we have used that $c_i = c_i^+ + c_i^-$, $\omega_{N_0}
c_i^{\pm} = \pm c_i^{\pm}$ and $\omega_X \omega_{N_0} = - \omega_{N_0}
\omega_X$. Hence by relation \eqref{relation 2},
\begin{equation}
  \nabla_{\nu} (\langle c_1, s \rangle \cdot \langle c_2, s \rangle) = -
  \langle N_0, N \rangle \langle c_1, s \rangle \cdot \langle c_2, s \rangle .
  \label{normal derivative}
\end{equation}

\end{proof}

\begin{lemma}
  We have
  \begin{equation}
    \langle c_1, s \rangle \cdot \langle c_2, s \rangle = \langle c_1,
    \omega_X s \rangle \cdot \langle c_2, \omega_X s \rangle \label{invariant on unit X}
  \end{equation}
  on $\Omega$ for all unit Euclidean $X$, $\omega_{N_0} c_i = c_i$.
\end{lemma}

\begin{proof}
  The identity \eqref{invariant on unit X} is valid obviously for $X = N_0$,
  it suffices to prove for $X$ normal to $X$.
  
  Along a face $F$ with the Euclidean normal $N$, since by Lemma \ref{4.8}, we have $\langle c_1, (1 +
  \omega_N) s \rangle \cdot \langle c_2, (1 - \omega_N) s \rangle = 0$. Using
  $N = a N_0 + b X$, we see
\begin{align}
& (1 - a^2) \langle c_1, s \rangle \cdot \langle c_2, s \rangle - b^2
\langle c_1, \omega_X s \rangle \cdot \langle c_2, \omega_X s \rangle
\\
& \quad - b (1 + a) \langle c_1, s \rangle \cdot \langle c_2, \omega_X s
\rangle + b \langle c_1, \omega_X s \rangle \cdot \langle c_2, s \rangle =
0.
\end{align}
  So
  \[ \langle c_1, s \rangle \cdot \langle c_2, s \rangle = \langle c_1,
     \omega_X s \rangle \cdot \langle c_2, \omega_X s \rangle \]
  along $F$. Let $f = \langle c_1, s \rangle \cdot \langle c_2, s \rangle -
  \langle c_1, \omega_X s \rangle \cdot \langle c_2, \omega_X s \rangle$, we
  see $\nabla^F f = 0$ along $F$. As calculated earlier
\begin{align}
\nabla_{\nu} (\langle c_1, s \rangle \cdot \langle c_2, s \rangle) & = -
\langle N_0, N \rangle \langle c_1, s \rangle \cdot \langle c_2, s
\rangle, \\
\nabla_{\nu} (\langle c_1, \omega_X s \rangle \cdot \langle c_2, \omega_X
s \rangle) & = - \langle N_0, N \rangle \langle c_1, \omega_X s \rangle
\cdot \langle c_2, \omega_X s \rangle,
\end{align}
  so $f = \nabla f = 0$ along $F$. Since we have that $\nabla_i \nabla_j f = f
  g_{i j}$ and we conclude that $f$ vanishes on all $\Omega$.

  Let $c_2 = \omega_Xc_3$ for some $c_3$ with $\omega_{N_0}c_3 = -c_3$, then
  \begin{equation}
    \label{invariant on unit X c3}
    \langle c_1, s \rangle \cdot \langle \omega_X c_3, s \rangle = \langle c_1, \omega_X s \rangle \cdot \langle c_3, s \rangle.
  \end{equation}
  The advantage of \eqref{invariant on unit X c3} over \eqref{invariant on unit X} is linearity with respect to $X$.
Because that \eqref{invariant on unit X c3} is valid for $N_0$ obviously and for
  all $N - \langle N_0, N \rangle N_0$ with $N$ being the normals of faces of
  $\Omega$, so \eqref{invariant on unit X c3} is valid for all vectors in $\mathbb{R}^n$. In
  particular, let $X$ be a unit vector,
  \[ \langle c_1, s \rangle \cdot \langle \omega_X c_3, s \rangle = \langle c_1, \omega_X s \rangle \cdot \langle c_3, s \rangle, \]
  By replacing $c_3$ back with $\omega_X c_2$ in \eqref{invariant on unit X c3}, we finish the proof of the lemma. 
\end{proof}

\begin{corollary} \label{cr:invariant under X}
  For any $c_1, c_2 \in \mathbb{C}^m$,
 \begin{equation}
    \langle c_1, s \rangle \cdot \langle c_2, \omega_X s \rangle = \langle
    c_1, \omega_X s \rangle \cdot \langle c_2, s \rangle \label{invariance under omega X}
  \end{equation}
  on $\Omega$ for all unit Euclidean vectors $X$.

\end{corollary}

\begin{proof}
  We have shown that the corollary holds for $\omega_{N_0} c_i = c_i$, and
  also when $\omega_{N_0} c_1 = c_1$, $\omega_{N_0} c_2 = - c_2$. The case
  $\omega_{N_0} c_i = - c_i$ is proved similarly. For the general case, we set
  $c_i^{\pm}$ such that $\omega_{N_0} c_i^{\pm} = \pm c_i^{\pm}$. Using
  Lemma \ref{lm:goal} and \eqref{invariant on unit X}, we see
\begin{align}
& \langle c_1, s \rangle \cdot \langle c_2, s \rangle \\
= & \langle c_1^+, s \rangle \cdot \langle c_2^+, s \rangle + \langle
c_1^-, s \rangle \cdot \langle c_2^-, s \rangle \\
= & \langle c_1^+, \omega_X s \rangle \cdot \langle c_2^+, \omega_X s
\rangle + \langle c_1^-, \omega_X s \rangle \cdot \langle c_2^-, \omega_X
s \rangle \\
= & \langle c_1, \omega_X s \rangle \cdot \langle c_2, \omega_X s \rangle
.
\end{align}
Hence, we have
  \begin{equation}
    \langle c_1, s \rangle \cdot \langle c_2, s \rangle = \langle c_1,
    \omega_X s \rangle \cdot \langle c_2, \omega_X s \rangle
  \end{equation}
  on $\Omega$ for all unit Euclidean $X$.
 Since $c_1$ and $c_2$ are arbitrary, we can replace $c_2$ with $\omega_Xc_2$. It then follows from self-adjointness of $\omega_X$, $\omega_X^2 = 1$ 
 that
  \begin{equation*}
    \langle c_1, s \rangle \cdot \langle c_2, \omega_X s \rangle = \langle
    c_1, \omega_X s \rangle \cdot \langle c_2, s \rangle 
  \end{equation*}
  on $\Omega$ for all unit Euclidean vectors $X$.
The corollary is proven.
\end{proof}

Now Proposition \ref{scalar identity} should be a simple consequence
of the previous corollary.

\begin{proof}[Proof of Proposition \ref{scalar identity}]
  Let $c_i = c^{(i)}_{\alpha}$ with $i = 1, 2$ and $\alpha = 1, \cdots, m$,
  let $G_{\alpha \beta} = \langle s_{\alpha}, s_{\beta} \rangle$. We write
  carefully \eqref{invariance under omega X} in components,
\begin{align}
\langle c_1, s \rangle \cdot \langle c_2, \omega_X s \rangle & = \langle
\bar{c}_{\alpha}^{(1)} s_{\alpha}, \bar{c}_{\mu}^{(2)} \omega_{X \mu
\lambda} s_{\lambda} \rangle = \bar{c}_{\alpha}^{(1)} G_{\alpha \lambda}
\bar{\omega}_{X \mu \lambda} c_{\mu}^{(2)}, \\
\langle c_1, \omega_X s \rangle \cdot \langle c_2, s \rangle & = \langle
\bar{c}_{\alpha}^{(1)} \omega_{X \alpha \lambda} s_{\lambda},
\bar{c}_{\mu}^{(2)} s_{\mu} \rangle = \bar{c}_{\alpha}^{(1)} \omega_{X
\alpha \lambda} G_{\lambda \mu} c_{\mu}^{(2)} .
\end{align}
  Since $c_1$ and $c_2$ are arbitrary, we know that $G_{\alpha \lambda}
  \bar{\omega}_{X \mu \lambda} = \omega_{X \alpha \lambda} G_{\lambda \mu}$.
  Since $\omega_X$ is self adjoint and $\omega_X^2 = 1$, we see
  $\bar{\omega}_{X \mu \lambda} = \omega_{X\lambda \mu}$ and $G_{\alpha
  \lambda} \omega_{X \lambda \mu} = \omega_{X \alpha \lambda} G_{\lambda
  \mu}$. This says that $G$ commutes with any $\omega_X$ where $X$ is of unit
  Euclidean length. So $G$ has to be a scalar multiple of the identity matrix.
\end{proof}

\subsection{Types of spinor components}

We see ${s}$ satisfies
\[ {\nabla}_{{\xi}} {s} - \tfrac{1}{2} {\epsilon} 
   {c} ({\xi}) \omega_{N_0} {s} = 0. \]
We fix the basis $\{\bar{s}_{\alpha} \}_{1 \leq \alpha \leq m}$ of
$\Delta_n$ such that $\b{\epsilon} \b{c} (N_0) \bar{s}_{\alpha} = \bar{s}_{\alpha}$ for $1
\leq \alpha \leq m / 2$ and  $\b{\epsilon} \b{c} (N_0) \bar{s}_{\alpha} = -
\bar{s}_{\alpha}$ for $m / 2 < \alpha \leq m$. We write \eqref{killing}
in components,
\begin{equation}
  {\nabla}_{{\xi}} {s}_{\alpha} - \tfrac{1}{2} {\epsilon}
  {c} ({\xi}) {s}_{\alpha} = 0\text{, } 1 \leqslant \alpha \leqslant \frac{m}{2}
  \text{ and }
  {\nabla}_{{\xi}} {s}_{\alpha} + \tfrac{1}{2} {\epsilon}
   {c} ({\xi}) {s}_{\alpha} = 0\text{, }  \frac{m}{2}<\alpha \leqslant m
  . \label{minus killing}
\end{equation}

\begin{lemma}\label{spinor-constant}
  \label{non-negativity of V}Let $V = |s_{\alpha} |^2$, then $V^2 - | \nabla
  V|^2$ is a non-negative constant.
\end{lemma}

\begin{proof}
  The proof is via direct calculation using \eqref{minus killing}. We show for
  $1 \leq \alpha \leq m / 2$. Let $p$ be a point in $\Omega$ and
  assume that $\{e_i \}$ is a geodesic normal frame at $p$, then $\nabla_{e_i}
  e_j = 0$ at $p$. We calculate the first and second derivatives of $V$.
  First,
\begin{align}
& \nabla_{e_i} V = \nabla_{e_i} \langle s_{\alpha}, s_{\alpha} \rangle
\\
= & \langle \nabla_{e_i} s_{\alpha}, s_{\alpha} \rangle + \langle
s_{\alpha}, \nabla_{e_i} s_{\alpha} \rangle \\
= & \tfrac{1}{2} \langle \epsilon c (e_i) s_{\alpha}, s_{\alpha}
\rangle + \tfrac{1}{2} \langle s_{\alpha}, \epsilon c (e_i)
s_{\alpha} \rangle \\
= & \langle \epsilon c (e_i) s_{\alpha}, s_{\alpha} \rangle .
\end{align}
  Hence
\begin{align}
& \nabla_{e_j} \nabla_{e_i} V \\
= & \nabla_{e_j} \langle \epsilon c (e_i) s_{\alpha}, s_{\alpha}
\rangle \\
= & \langle \epsilon c (e_i) \nabla_{e_j} s_{\alpha}, s_{\alpha}
\rangle + \langle \epsilon c (e_i) s_{\alpha}, \nabla_{e_j}
s_{\alpha} \rangle \\
= & \tfrac{1}{2} \langle \epsilon c (e_i) \epsilon c
(e_j) s_{\alpha}, s_{\alpha} \rangle + \tfrac{1}{2} \langle \epsilon
 c (e_i) s_{\alpha}, \epsilon c (e_j) s_{\alpha} \rangle
\\
= & - \tfrac{1}{2} \langle c (e_i) c (e_j) s_{\alpha} \rangle -
\tfrac{1}{2} \langle c (e_j) c (e_i) s_{\alpha}, s_{\alpha} \rangle
\\
= & |s_{\alpha} |^2 \delta_{i j} = V \delta_{i j},
\end{align}
  where we have used the simple facts that $- 2 \delta_{i j} = c (e_i) c (e_j)
  + c (e_j) c (e_I)$, $\epsilon c (e_i) = - c (e_i) \epsilon$ and
  $\epsilon^2 = 1$. So
\begin{align}
& \nabla_{e_i} (V^2 - | \nabla V|^2) \\
= & 2 V \nabla_i V - 2 \sum_j \nabla_{e_j} \nabla_{e_i} V \nabla_{e_j} V
\\
= & 2 V \nabla_i V - 2 \sum_j V \delta_{i j} \nabla_{e_j} V = 0
\end{align}
  and hence $V^2 - | \nabla V|^2$ is a constant. To check that $V^2 - | \nabla
  V|^2$ is non-negative, we calculate the squared length of the spinor
  $\epsilon s_{\alpha} - V^{- 1} \sum_j \nabla_j V c (e_j) s_{\alpha}$ is
  precisely $V^2 - | \nabla V|^2$. Indeed,
\begin{align}
& | \epsilon s_{\alpha} - V^{- 1} \sum_j \nabla_j V c (e_j) s_{\alpha}
|^2 \\
= & \langle \epsilon s_{\alpha}, \epsilon s_{\alpha} \rangle - V^{-
1} \sum_j \nabla_j V \langle \epsilon s_{\alpha}, c (e_j) s_{\alpha}
\rangle \\
& \quad - V^{- 1} \sum_j \nabla_j V \langle c (e_j) s_{\alpha},
\epsilon s_{\alpha} \rangle \\
& \quad + V^{- 2} \sum_{j, k} \nabla_j V \nabla_k
V\ensuremath{\operatorname{Re}} \langle c (e_j) s_{\alpha}, c (e_k)
s_{\alpha} \rangle .
\end{align}
  Here $\ensuremath{\operatorname{Re}}$ denotes taking the real part. Since
  $\nabla_{e_i} V = \langle \epsilon c (e_i) s_{\alpha}, s_{\alpha}
  \rangle$, so
  \[ \langle c (e_j) s_{\alpha}, \epsilon s_{\alpha} \rangle = \langle
     \epsilon s_{\alpha}, c (e_j) s_{\alpha} \rangle = \nabla_j V \]
  Note that $\langle \epsilon s_{\alpha}, \epsilon s_{\alpha} \rangle =
  |s_{\alpha} |^2 = V$ and
  \[ \ensuremath{\operatorname{Re}} \langle c (e_j) s_{\alpha}, c (e_k)
     s_{\alpha} \rangle = |s_{\alpha} |^2 \delta_{i j} = V \delta_{i j}, \]
  we obtain that
  \[ | \epsilon s_{\alpha} - V^{- 1} \sum_j \nabla_j V c (e_j) s_{\alpha}
     |^2 = V^2 - | \nabla V|^2,\]
  which is obviously non-negative.
\end{proof}

\begin{definition}\label{defn-typeIII}
  Let $\phi$ be a spinor which satisfies
  \begin{equation}
    \nabla_{e_i} \phi \pm \tfrac{1}{2} \epsilon c (e_i) \phi = 0,
    \label{general killing}
  \end{equation}
  and $V_{\phi} = | \phi |^2$. We say that $\phi$ is of type I if $V_{\phi} =
  0$ and of type II if $V_{\phi} > 0$.
\end{definition}
\begin{remark}\label{rem-typeI}
    The proof of Lemma \ref{spinor-constant} is the same as the proof of \cite[Lemma 5]{Baum1989}. According to Lemma \ref{spinor-constant}, a spinor $\phi$ is of type I if and only if there exists a unit vector $\nu_0\in T_pM$ such that $\epsilon c(\nu_0)\phi=\phi$.
\end{remark}

By the proof of Lemma \ref{non-negativity of V}, we have that $\phi$ which
satisfies \eqref{general killing} is of type I if and only if
\begin{equation}
  \epsilon \phi = \mp c (\nabla \log V) \phi.
  \label{equivalent type I}
\end{equation}

\begin{lemma}
  \label{type I implies parallel}If a spinor $\phi$ satisfies \eqref{general killing}
  and is of type I, let $F$ be a level set of $V = | \phi |^2$, then $\phi
  |_F$ is a parallel spinor.
\end{lemma}

\begin{proof}
  Let $e_n = \nabla \log V$, since $\phi$ is of type I, $e_n$ is a unit normal
  to $F$. Let $X$ be any vector field, then
\begin{align}
& \nabla_X e_n \\
= & \nabla_X \nabla \log V \\
= & \nabla_X (V^{- 1} \nabla V) \\
= & - V^{- 2} \nabla_X V \nabla V + V^{- 1} \nabla_X \nabla V \\
= & - \langle X, e_n \rangle e_n + X.
\end{align}
  The induced connection on $F$ is given by
  \[ \nabla^{\partial}_{e_i} \phi = - c (e_n) \nabla_{e_i} \phi - \tfrac{1}{2}
     A_{i j} c (e_j) \phi, \]
  where $A$ is the second fundamental form of $F$ in $\Omega$. Since $A_{i j}
  = \langle \nabla_{e_i} e_n, e_j \rangle = \delta_{i j}$,
  \[ \nabla^{\partial}_{e_i} \phi = - c (e_n) \nabla_{e_i} \phi - \tfrac{1}{2}
     c (e_i) \phi . \]
  Using \eqref{general killing} and \eqref{equivalent type I}, we see
  \[ \nabla^{\partial}_{e_i} \phi = \pm \tfrac{1}{2} c (e_n) \epsilon c
     (e_i) \phi - \tfrac{1}{2} c (e_i) \phi = 0. \]
  Hence, $\phi |_F$ is a parallel spinor.
\end{proof}

\subsection{Proof of Theorem \ref{intersect rigidity}}

Now we finish the proof of Theorem \ref{intersect rigidity}. 
First, we show every face of $\Omega$ is umbilic.
\begin{lemma} \label{curvature of boundary}
The principal curvatures of a face are
  constant and equal to $\langle \partial_{x^1}, N \rangle$.
\end{lemma}

\begin{proof}
Let $e_i$ be an orthonormal frame such that $e_n = \nu$ and the second
fundamental form $A$ of a face $F$ in $\Omega$ is diagonalized at some point
$x \in F$, that is, $A_{ij} = \kappa_i \delta_{ij}$ where $i, j \neq n$.

By the boundary condition on $F$, $\omega_N s = - \epsilon c (\nu) s$. By
differentiating $\omega_N \epsilon c (\nu) s$ and using $\omega_N s = -
\epsilon c (\nu) s$, we see
\begin{align}
  & \nabla_{e_i} (\omega_N \epsilon c (\nu) s) \\
  = & \omega_N \epsilon c (\nabla_{e_i} \nu) s + \omega_N \epsilon c
  (\nu) \nabla_{e_i} s \\
  = & - \kappa_i \omega_N \epsilon c (e_i) s + \tfrac{1}{2} \omega_N
  \epsilon c (\nu) \omega_{N_0} \epsilon c (e_i) s \\
  = & - \kappa_i \omega_N \epsilon c (e_i) s - \tfrac{1}{2} (2 \langle N,
  N_0 \rangle - \omega_{N_0} \omega_N) \epsilon c (e_i) \epsilon c (\nu)
  s \\
  = & - \kappa_i \omega_N \epsilon c (e_i) s + \tfrac{1}{2} (2 \langle N,
  N_0 \rangle - \omega_{N_0} \omega_N) \epsilon c (e_i) \omega_N s
  \\
  = & - \kappa_i \omega_N \epsilon c (e_i) + \langle N, N_0 \rangle
  \omega_N \epsilon c (e_i) s - \tfrac{1}{2} \omega_{N_0} \epsilon c
  (e_i) s 
\end{align}
By differentiating $s$, $- \nabla_{e_i} s = - \tfrac{1}{2} \omega_{N_0}
\epsilon c (e_i) s$ and using the boundary condition again, we conclude that
\[ - \kappa_i \omega_N \epsilon c (e_i) s + \langle N, N_0 \rangle \omega_N
   \epsilon c (e_i) s = 0. \]
As $s$ has at least one nonzero component, we see \( \kappa_i = \langle N, N_0 \rangle . \)
\end{proof}

Returning to our problem regarding Theorem \ref{intersect rigidity}.

\begin{proof}[Proof of Theorem \ref{intersect rigidity}]
  The principal curvatures of faces $F_{\ell}$ are given by Lemma
  \ref{curvature of boundary}.
  It follows from Proposition \ref{lm:goal} that the components of $s$ are linearly independent, hence, we can use similar arguments as
  in the proof of Theorem \ref{hyperbolic Llarull} to show that $(\Omega,g)$ is hyperbolic. It also follows from Proposition \ref{lm:goal} that all $s_\alpha$ are of type I. 
 By Lemma \ref{type I implies parallel},
  $s_{\alpha}$ restricted on its level set are parallel and hence its level
  sets are flat. We can pick a coordinate such that $g = \tfrac{1}{(x^1)^{2}}
  ((\mathrm{d} x^1)^2 + \cdots + (\mathrm{d} x^n)^2)$ in some set $\Omega
  \subset \mathbb{R}_+^n$ such that level sets of $|s_{\alpha} |$ lies in
  $x^1$-coordinate hyperplane. By Lemma \ref{curvature of boundary}, every
  face is umbilic, so it is either part of a sphere or a linear hyperplane. By
Lemma \ref{hessian and boundary derivative}, $\nabla_{\nu_{\ell}} |s_{\alpha} |^2 = \langle N_{\ell},
  \tfrac{\partial}{\partial x^1} \rangle  |s_{\alpha} |^2$, hence every
  face has to be a part of a linear hyperplane. Because if the face were a part of a sphere, the value of
$\nabla_{\nu_{\ell}} |s_{\alpha} |^2 / |s_{\alpha} |^2$ would not be a constant.
   We can conclude now that $(\Omega,g)$ is a polytope in some Poincar\'{e} half-space model.
\end{proof}

\section{Scalar curvature rigidity in odd-dimensional hyperbolic space} \label{odd D}

In this section, we address the odd-dimensional case of the scalar curvature rigidity results, specifically Theorems \ref{hyperbolic Llarull} and \ref{intersect rigidity}, in hyperbolic space. Most of the proofs are similar to the even-dimensional
case. We only highlight the main differences and leave the details.

Let $\sigma \in S_{\Omega_g} \otimes S_{\Omega_{\delta}}$, we consider the
following connection
\begin{equation}
  \hat{\nabla}_{e_i} \sigma = \nabla_{e_i} \sigma + \tfrac{\sqrt{- 1}}{2} c
  (e_i) \otimes (\sqrt{- 1} \bar{c} (N_0)) \sigma \label{twisted imaginary killing}
\end{equation}
and its associated Dirac operator is given by
\begin{equation} \label{twisted imaginary Dirac}
  \hat{D} \sigma = c (e_i) \hat{\nabla}_{e_i} \sigma .
  \end{equation}
We introduce the local boundary condition
\begin{equation} \label{twisted imaginary chi}
  \chi \sigma = (\sqrt{- 1} c (e_n) \otimes \sqrt{- 1} \bar{c} (N))  \sigma,
  \end{equation}
where $e_n$ and $N$ are respectively the unit normal of $\partial \Omega$ in
$\Omega$ with respect to the metric $g$ and the flat metric. Analogous to
Proposition \ref{Dirac Witten}, we have the
Schrodinger-Lichnerowicz formula whose proof we shall omit.

\begin{proposition}
Let $\sigma \in S_{\Omega_{g}} \otimes S_{\Omega_{\delta}}$, then
  \begin{equation}\label{imaginary boundary dirac chi anti}
    D^{\partial} \chi + \chi D^{\partial} = 0,
  \end{equation}
  (that is, $D^{\partial}$ and $\chi$ anti-commute) and
\begin{align}\label{eqn-D imaginary}
& \int_N | \tilde{D} \sigma |^2 \\
  = & \int_N | \tilde{\nabla} \sigma |^2  
      + \int_N \tfrac{1}{4} (R_g + n(n-1)) |\sigma|^2  \\
      & \quad + \int_{\partial N}
     \tfrac{1}{4} \langle D^{\partial} (\sigma + \chi \sigma), \sigma - \chi
\sigma \rangle + \tfrac{1}{4} \langle D^{\partial} (\sigma - \chi \sigma),
      \sigma + \chi \sigma \rangle \\
     &  \quad+  \int_{\partial N} \langle \mathcal{A} \sigma,\sigma \rangle + \tfrac{n - 1}{2}  \langle c (e_n) \bar{c}(N_0) \sigma, \sigma \rangle,
 \end{align}
where 
\begin{equation}\label{imaginary A}
  \mathcal{A}:= \tfrac{1}{2} H_g - \tfrac{1}{2} \sum_{1 \leq i \leq n-1} c (e_n) c
  (e_i) \bar{c} (\bar{\nabla}_{e_i} \bar{e}_n) \bar{c} (\bar{e}_n).
  \end{equation}
\end{proposition}
\begin{remark} \label{brendle index}
  Using the index theory of \cite[Proposition 2.15]{brendle-scalar-2024}, the operator
  \[ \sigma \mapsto (D \sigma, \sigma - \chi \sigma) \]
  is of Fredholm index 1 (note that we have omitted the underlying Sobolev space). Moreover,
  if we replace the Dirac operator $D$ with $\hat{D}$, or if we replace the Gauss map
  $N$ in $\chi$ with a map homotopic to the Gauss map, the index is still 1.
  \end{remark}
  
\begin{remark} \label{rem:imaginary trace form}
   We can also obtain the formula \eqref{sl for polytope epsilon} using
   $\hat{\nabla}$, $\hat{D}$ and $\chi$ defined \eqref{twisted imaginary killing},
   \eqref{twisted imaginary Dirac} and \eqref{twisted imaginary chi}.
  \end{remark}

Now we give an alternative proof of Theorem \ref{hyperbolic Llarull} 
using simple linear algebra.

\begin{lemma}
  Let $V$ and $W$ be finite-dimensional real vector spaces of the same
  dimension. The space $W$ is equipped with an inner product and $V$ is
  equipped with two inner products $G_1$ and $G_2$. Let $L : V \to W$ be a
  linear isomorphism, the trace norms of $L$ are defined by
  $\|L\|_{\ensuremath{\operatorname{tr}}, i} = \sup_Q
  \ensuremath{\operatorname{tr}} (Q L)$, where the supremum is taken over all
  linear isometries $Q : W \to (V, G_i)$. If $G_2 \geqslant G_1$, then
  $\|L\|_{\ensuremath{\operatorname{tr}}, 2} \leqslant
  \|L\|_{\ensuremath{\operatorname{tr}}, 1}$. Equality is achieved if and only if
  $G_2 = G_1$.
\end{lemma}

\begin{proof}
  Let $\{e_i \}_{1 \leqslant i \leqslant n}$ be a basis of $V$ such that $G_1
  (e_i, e_j) = \delta_{ij}$ and $G_2 (e_i, e_j) = \mu_i \delta_{ij}$. Since
  $G_2 \geqslant G_1$, so $\mu_i \geqslant 1$ for all $1 \leqslant i \leqslant
  n$. Let $Q : W \to (V, G_1)$, then $SQ$ is an isometry from $W$ to $(V,
  G_2)$ where $S : V \to V$ is a linear map given by sending all $e_i$ to
  $\tfrac{1}{\sqrt{\mu_i}} e_i$.
  
  We fix an orthonormal basis $\{ \hat{E}_i \}_{1 \leqslant i \leqslant n}$ of
  $W$, now we can view maps between $W$ and $V$ as $\ell \times \ell$
  matrices, then a map $Q : W \to V$ given by
  \[ \hat{E}_i \mapsto Q \hat{E}_i := \sum_j Q_{ij} e_j \]
  is an isometry from $W$ to $(V, G_1)$ if and only if $\{Q_{ij} \}$ is an
  orthogonal matrix which we still denote by $Q$. We set $S = \mathrm{diag}
  (\tfrac{1}{\sqrt{\mu_1}}, \ldots, \tfrac{1}{\sqrt{\mu_n}})$, then $SQ$
  represents an isometry from $W$ to $(V, G_2)$. By definition of the trace norm,
  \[ \|L\|_{\mathrm{tr}, 1} = \sup_{Q \in O (n)} \mathrm{tr} (QL),
     \|L\|_{\mathrm{tr}, 2} = \sup_{Q \in O (n)} \mathrm{tr} (SQL) . \]
  Take an arbitrary orthogonal matrix $Q \in O (n)$, let $\lambda_i$ be the
  $i$-th diagonal entry of $QL$, then the $i$-th diagonal entry of $SQL$ is
  $\lambda_i / \sqrt{\mu_i}$. So
  \[ \mathrm{tr} (SQL) = \sum_i \lambda_i / \sqrt{\mu_i} \leqslant \sum_i |
     \lambda_i | = \mathrm{tr} (S' QL), \]
  where $S'$ is a suitable diagonal matrix depending on $Q$ with diagonal
  entries $1$ or $- 1$ such that all the diagonal entries of $S' QL$ are
  nonnegative. Note that $S' Q$ is also an orthogonal matrix.
  
  By definition of the trace norm,
  \[ \|L\|_{\mathrm{tr}, 2} = \sup_{Q \in O (n)} \mathrm{tr} (SQL) \leqslant
     \sup_{Q \in O (n)} \mathrm{tr} (S' QL) =\|L\|_{\mathrm{tr}, 1} . \]
  Since $L$ is a linear isomorphism, $\lambda_i \neq 0$. We easily find that
  the equality holds if and only if $\mu_i = 1$ for all $i$, that is, $G_2 =
  G_1$.
\end{proof}

\begin{proof}[Alternative proof of Theorem \ref{hyperbolic Llarull}]
  Using Remark \ref{brendle index}, we solve $\hat{{D}} s = 0$ subject to the boundary condition
  $\omega_N \sqrt{- 1} c(\nu) s = s$, then
\begin{align}
0 = & \int_{\Omega} | \hat{{D}} s|^2 \\
\geqslant & \int_{\Omega} | \hat{\nabla} s|^2 + \tfrac{1}{4} (R_g + n (n
- 1)) |s|^2 + \tfrac{1}{2} \int_{\Sigma} (H + (n - 1) \mathrm{d} x^1 (N)
-\| \mathrm{d} N\|_{\ensuremath{\operatorname{tr}}}) |s|^2 \\
\geqslant & \tfrac{1}{2} \int_{\partial \Omega} (H_{b} + (n - 1)
\mathrm{d} x^1 (N) -\| \mathrm{d} N\|_{\ensuremath{\operatorname{tr}},
\sigma}) |s|^2 .
\end{align}
  Note that $H_{b} = x^1 (H_{\delta} + (n - 1) \partial_N \log
  \tfrac{1}{x^1})$, so
  \[ 0 \geqslant \tfrac{1}{2} \int_{\partial \Omega} (x^1 H_{\delta} -\|
     \mathrm{d} N\|_{\ensuremath{\operatorname{tr}}, g|_{\partial \Omega}}) |s|^2 . \]
   We see that by $ g|_{\partial \Omega} \geqslant b|_{\partial \Omega}$ that $g|_{\partial \Omega}\geqslant (\tfrac{1}{x^1})^2
  \delta|_{\partial \Omega}$. By the previous lemma, we have that
  \[ \| \mathrm{d} N\|_{\ensuremath{\operatorname{tr}}, \sigma} \leqslant
    \| \mathrm{d} N\|_{\ensuremath{\operatorname{tr}}, \tfrac{1}{(x^1)^{2}} \delta|_{\partial \Omega}}
    = x^1 \| \mathrm{d} N\|_{\ensuremath{\operatorname{tr}}, \delta|_{\partial \Omega}}
    = x^1 H_{\delta}, \]
  which forces $\hat{\nabla} s = 0$ and $R_g + n (n - 1) = 0$. The rest of
  the argument is the same as the even-dimensional case.
\end{proof}

Now we prove the odd-dimensional case of Theorem \ref{intersect rigidity}.

\begin{proof}[Proof of Theorem \ref{intersect rigidity} in odd dimensions]
  We use again Brendle's smoothing \eqref{brendle smoothing explicit} $\Omega_{\lambda}$.
Using the index theory (Remark \ref{brendle index}), we can
solve the following problem
\[ \hat{D} \sigma^{(\lambda)} = 0 \text{ in } \Omega_{\lambda}, \text{ } \chi_{\lambda}
   \sigma^{(\lambda)} = \sigma^{(\lambda)}\text{ along } \partial \Omega_{\lambda}  \]
as the even-dimensional case
in
a sequence of approximating domains $\Omega_{\lambda}$.
Because the Schrodinger-Lichnerowicz formula \eqref{sl for polytope epsilon} (see Remark \ref{rem:imaginary trace form}) holds for smooth domains $\Omega_{\lambda}$ in the
same form, 
there exists a subsequence $\{\lambda_{l}\}_{l\in \mathbb{N}}$
such that as $l \to \infty$, we see that the map $N_{\lambda_l}$ tends to the normal $N$ of $\partial \Omega$ and the solution
$\sigma^{(\lambda_{l})}$ converge to a nonzero section $\sigma \in S_{\Omega_g}\otimes S_{\Omega_{\delta}}$ with
  \begin{equation}
  \hat{\nabla} \sigma = 0 \text{ in } \Omega, \text{ } \chi \sigma = \sigma
  \text{ along } \partial \Omega .
\end{equation}
We interpretate again the section $\sigma\in S_{\Omega_g} \otimes S_{\Omega_{\delta}}$ in terms of $m$-tuple of spinors $s$. As before, we can show that
the components of $s$ are linearly independent and
  \[ \sum_{i, j=1}^n (R (e_i, e_j, e_k, e_l) + \delta_{i k} \delta_{j l} - \delta_{i l}
    \delta_{j k}) c (e_i) c (e_j) s_{\mu} = 0 \] for any $1 \leqslant \mu \leqslant m$.
Since the
dimension is odd, the kernel of the spinor representation
$\ensuremath{\operatorname{Cl}} (T_x \Omega) \to
\ensuremath{\operatorname{End}} (S_{\Omega_g})$ is given by the $(-
1)$-eigenspace of the complex volume form $\Gamma = (\sqrt{- 1})^{\tfrac{n
+ 1}{2}} c (e_1) \cdots c (e_n) \in \ensuremath{\operatorname{Cl}} (T_x \Omega)$
(see Theorem 1.28 and Definition 1.31 of {\cite{BHMMM}}). Hence
  \[ \sum_{i < j} (R (e_i, e_j, e_k, e_l) + \delta_{i k} \delta_{j l} - \delta_{i l}
\delta_{j k}) c (e_i) c (e_j) (1 + \Gamma) c (e_i) c (e_j) = 0 \] and it follows
that
  \[ R (e_i, e_j, e_k, e_l) + \delta_{i k} \delta_{j l} - \delta_{i l} \delta_{j k} =
0. \] So $(\Omega, g)$ is hyperbolic. The calculation of the principal
curvatures of the boundary is the same as Lemma \ref{curvature of
boundary} (replacing only $\epsilon$ by $\sqrt{-1}$).
  
Recall that a spinor $s_{\lambda}$ is of type I if there exists a unit Euclidean
vector $e \in T_x \Omega$ such that $\sqrt{- 1} c (e) s_{\lambda} = s_{\lambda}$, see Remark \ref{rem-typeI}. As in
Proposition \ref{lm:goal}, we see that $s_{\lambda}$ is of type I for every
$1 \leqslant \lambda \leqslant m$. Let $V = |s_{\lambda} |^2 =:
\tfrac{1}{x^1}$, then $V$ gives the $x^1$-coordinate of the Poincar{\'e} half
space model which $\Omega$ lies in and we show similarly that $(\Omega, g)$ is a
polytope in this model.
\end{proof}

\bibliographystyle{alpha}
\bibliography{llarull-boundary}

\end{document}